\documentclass[11pt,fleqn,a4paper]{article}

\pdfoutput=1

\usepackage{pdfsync}
\usepackage{graphicx}
\usepackage{latexsym}
\usepackage[small,bf]{caption}
\usepackage{amsfonts}
\usepackage{amssymb}
\usepackage{amsmath}
\usepackage{stmaryrd}
\usepackage{bbm}
\usepackage{times}
\usepackage{color}

\usepackage{hyperref}
\usepackage{authblk}

\usepackage[]{algorithm2e}

\usepackage[top=3cm, bottom=3cm, outer=3cm, inner=3cm, heightrounded, marginparwidth=1.9cm, marginparsep=0.3cm]{geometry}
\usepackage{amsthm}
\makeatother
\newtheorem{thm}{Theorem}

\newtheorem{prop}{Proposition}
\newtheorem{cor}{Corollary}
\theoremstyle{definition}
\newtheorem{defn}{Definition}

\newcommand{\N}{\mathbb{N}}
\newcommand{\R}{\mathbb{R}}

\newcommand{\Lcal}{\mathcal{L}}
\newcommand{\1}{\mathbbm{1}}

\begin{document}

\title{Structure of the condensed phase in the inclusion process}

\author[1]{Watthanan Jatuviriyapornchai}
\author[2]{Paul Chleboun}
\author[3]{Stefan Grosskinsky}
\affil[1]{\small Department of Mathematics, Faculty of Science, Mahidol University, Bangkok, 10400, Thailand}
\affil[1]{Centre of Excellence in Mathematics, Commission on higer Education, Bangkok, 10400, Thailand}
\affil[2]{Department of Statistics, University of Warwick, Coventry CV4 7AL, UK}
\affil[3]{Mathematics Institute, University of Warwick, Coventry CV4 7AL, UK}

\maketitle

\normalsize

\begin{abstract}
We establish a complete picture of condensation in the inclusion process in the thermodynamic limit with vanishing diffusion, covering all scaling regimes of the diffusion parameter and including large deviation results for the maximum occupation number. We make use of size-biased sampling to study the structure of the condensed phase, which can extend over more than one lattice site and exhibit an interesting hierarchical structure characterized by the Poisson-Dirichlet distribution. While this approach is established in other areas including population genetics or random permutations, we show that it also provides a powerful tool to analyse homogeneous condensation in stochastic particle systems with stationary product distributions. We discuss the main mechanisms beyond inclusion processes that lead to the interesting structure of the condensed phase, and the connection to other generic particle systems. Our results are exact, and we present Monte-Carlo simulation data and recursive numerics for partition functions to illustrate the main points.
\end{abstract}


\section{Introduction}

Condensation phenomena in stochastic particle systems (SPS) continue to be a topic of major research interest. They can be caused by spatial inhomogeneities (see e.g.\ \cite{godreche2012condensation,chleboun2014condensation} and references therein) or attractive particle interaction in spatially homogeneous systems, which is the focus of this paper. If the total density of particles exceeds a critical value, the system phase separates into a homogeneous bulk and a condensed phase, with a finite fraction of the total mass concentrating in a vanishing volume fraction. 
First introduced in \cite{spitzer1970interaction}, zero-range processes and related models provided a first example of condensation in homogeneous SPS \cite{drouffe1998simple,evans2000phase,godreche2003dynamics}. 
On the level of stationary distributions condensation is characterized by heavy-tail behaviour of stationary weights as first noted in  \cite{jeon2000size,jeon2000condensation}, which has been used to study the phenomenon in the context of equivalence of ensembles and large deviations \cite{grosskinsky2003condensation, armendariz2009thermodynamic,armendariz2011conditional}.

The inclusion process has been introduced in \cite{giardina2007duality} as a discrete dual to a model of heat conduction, and has later been studied as an interesting model of stochastic transport on its own \cite{giardina2009duality,giardina2010correlation, carinci2013}. 
It is a natural bosonic counterpart to the exclusion process where particles are subject to an attractive inclusion interaction in addition to independent diffusive motion. It can also be interpreted as a multi-species version of the Moran model of population genetics \cite{moran1958random}, where the inclusion interaction corresponds to selection, and diffusion to mutation dynamics. The inclusion process is part of a larger class of models introduced in \cite{cocozza1985processus} that exhibit factorized stationary distributions, which has recently been extended \cite{fajfrova2016invariant}. Condensation in the inclusion process has first been studied in \cite{grosskinsky2011condensation} for inhomogeneous systems. Condensation in homogeneous systems only occurs if the diffusion strength vanishes with the system size. While such scaling of system parameters can lead to non-equivalence of ensembles and discontinuous behaviour as established for a toy zero-range model in \cite{grosskinsky2008discontinuous,chleboun2015dynamical}, this is not the case for the inclusion process and small diffusion or mutation rates are in fact very natural in many applications. The dynamics on various time scales have been established on a rigorous level in \cite{grosskinsky2013dynamics,bianchi2017metastability}, restricted to finite lattices in the limit of diverging particle density. In the thermodynamic limit with a finite limiting density there are only heuristic results so far, covering the dynamics of condensation in the inclusion process \cite{cao2014dynamics} and extensions with stronger particle interactions and instantaneous condensation \cite{waclaw2012explosive,chau2015explosive}. 

In particular, the stationary behaviour of the inclusion process in the thermodynamic limit has not been characterized so far, which is the main aim of this paper. We establish the equivalence of ensembles, and show that for vanishing diffusion strength the inclusion process exhibits condensation for any positive particle density. 
While the bulk of the system is empty, the condensed phase can exhibit an interesting hierarchical structure following the Poisson-Dirichlet distribution. The latter was originally introduced in the context of population genetics  \cite{kingman1975random,kingman1977population}, and has later been identified as the generic stationary distribution of split-merge dynamics \cite{pitman2002,diaconis2004poisson}, which is related to its appearance in cycle length distributions of random permutations \cite{berestycki2011emergence,betz2011spatial, grosskinsky2012lattice}. 
It has further been observed (though not identified) more recently in systems of interacting diffusions 
\cite{jack2017emergence,andres2010particle}, but to our knowledge is a novelty in the context of condensation in SPS. 
In general, the condensed phase in SPS with stationary product distributions concentrates on a single lattice site \cite{jeon2000size,jeon2000condensation, armendariz2009thermodynamic,armendariz2013zero}. A spread over multiple sites has only been observed in versions of zero-range processes which include an effective (soft) cut-off for site occupation numbers  \cite{schwarzkopf2008zero,thompson2010zero}, or in models with pair-factorized stationary states \cite{evans2006interaction,waclaw2009pair} where it occurs naturally due to spatial correlations. Poisson-Dirichlet statistics arise when the diffusion parameter in the inclusion process scales with the inverse system size, and we also establish complete condensation for smaller diffusion where all particles concentrate on a single site, and a universal exponential law for intermediate scales.

Our main results on the structure of the condensed phase are derived using size-biased sampling of occupation numbers, which is related in a natural way to the Poisson-Dirichlet distribution as reviewed in Section \ref{sec:pd}. 
While this point of view is standard in population genetics (see e.g. \cite{feng2010poisson}), this approach also provides a strong tool to study the condensed phase in SPS where it has not been used so far. After introducing the basic notation and concepts in Section \ref{sec:setting}, we derive our main results on condensation and the typical structure of the condensed phase for the inclusion process in Section \ref{sec:main}. Our results are rigorous and derivations are presented in a general, transferable way, and we show simulation data for illustration. We include results on large deviations of the condensed phase in Section \ref{sec:ld}, and conclude with a discussion of the main points and relations to other models in Section \ref{sec:dis}. In Appendix \ref{appa} we show that under a general definition of condensation the system phase separates into a homogeneous bulk and a condensed phase, and that condensation implies divergence of higher moments. In Appendix \ref{appb} we comment on Monte-Carlo dynamics to generate stationary samples, and on differences between one-dimensional and mean-field geometries.

\section{Mathematical setting\label{sec:setting}}

\subsection{Condensation in homogeneous particle systems\label{sec:condensation}}

We study stochastic particle systems (SPS) on a finite set of spatial locations/sites $\Lambda$ of size $|\Lambda |=L$, which can for example be a regular lattice with periodic or closed boundaries. The system has a fixed, finite number of $N$ particles, and we denote configurations by $\eta =(\eta_x :x\in\Lambda )$, $\eta_x \in\N_0$, and the state space $E_{L,N} =\big\{\eta :\sum_{x\in\Lambda} \eta_x =N\big\}$ denotes the set of all configurations. The dynamics should be irreducible on $E_{L,N}$, so that the process has a unique (canonical) stationary distribution $\pi_{L,N}$. We assume that $\pi_{L,N}$ is spatially homogeneous, i.e. the single-site marginals $\pi_{L,N} [\eta_x \in .]$ do not depend on site $x$, and in particular this implies that the density (the expected number of particles per site) is given as
\begin{equation}\label{themoment}
\langle \eta_x \rangle_{L,N} := \sum_{n=1}^N n\,\pi_{L,N} [\eta_x =n ] =N/L\ .
\end{equation}

We are interested in large-scale condensation phenomena of the system in the thermodynamic limit $L,N\to\infty$ such that the density converges as $N/L\to\rho \geq 0$, which in the following we often denote by $\lim_{N/L\to\rho}$ to simplify notation. We assume that in this limit finite marginals of $\pi_{L,N}$ converge, and we denote the limiting single site marginal as a distribution on $\N_0$ by
\begin{equation}\label{margi}
\nu_\rho :=\lim_{N/L\to\rho} \pi_{L,N} [\eta_x \in .]\ .
\end{equation}
This convergence of distribution functions is equivalent to weak convergence, i.e.
\begin{equation}\label{weak}
\big\langle f(\eta_x )\big\rangle_{L,N} \to \langle f\rangle_\rho \quad\mbox{as }L,N\to\infty\ ,\quad N/L\to\rho
\end{equation}
for all $x\in\Lambda$ and bounded, continuous test functions $f\in C_b (\N_0 )$. 
With \eqref{themoment} the first moment $\langle \eta_x\rangle_{L,N}\to\rho$ converges in the thermodynamic limit, and by Fatou's Lemma this implies for the first moment of the limiting distribution that
\begin{equation}\label{rhobg}
\rho_{b} :=\langle \eta_x \rangle_\rho \leq\rho\ .
\end{equation}
This is usually called the \textbf{background} or \textbf{bulk density} (indicated by the subscript) as is explained below. Strict inequality above is possible since $f(\eta_x )=\eta_x$ is an unbounded function on $\N_0$, and implies that locally the system loses mass in the limit, providing the following standard definition of condensation. \\

\begin{defn}\label{cdef}
A system with canonical distributions $\pi_{L,N}$ exhibits \textbf{condensation} in the thermodynamic limit $N/L\to\rho$ with background density $\rho_{b}$ as in \eqref{rhobg}, if $\nu_\rho$ exists as defined in \eqref{margi} and $\rho_b <\rho$. A system with $\rho_b =0$ is said to exhibit \textbf{complete condensation} if
\begin{equation}\label{complete}
\pi_{L,N} \Big[\max_{x\in\Lambda} \eta_x =N\Big] \to 1\quad\mbox{as }L,N\to\infty ,\ N/L\to\rho\ ,
\end{equation}
i.e. typically all particles in the system concentrate on a single lattice site.\\
If $\nu_\rho$ exists for all $\rho\geq 0$, the systems is said to exhibit a \textbf{condensation transition} with \textbf{critical density} $\rho_c \geq 0$, if
\begin{equation}\label{rhoc}
\rho_b \left\{\begin{array}{cl} =\rho&,\mbox{ for all }\rho <\rho_c\\ <\rho&,\mbox{ for all }\rho >\rho_c\end{array}\right.\ .
\end{equation}
\end{defn}

Condensation in the above setting has been established in various SPS, including zero-range processes and related models (see e.g. \cite{evans2005, evans2014condensation} and references therein). 
It has been shown on a case-by-case basis that $\rho_b$ is monotone increasing with $\rho$ and there exists a unique critical density $\rho_c \in [0,\infty]$ in the sense of \eqref{rhoc}. One sufficient general condition is monotonicity of the dynamics for the underlying particle system. But in principle more complicated behaviour such as non-monotonicity of $\rho_b$ cannot be ruled out, even though we are not aware of any generic examples in the thermodynamic limit. For condensation on finite lattices possible non-monotonicity of $\rho_b$ has been established and discussed e.g. in \cite{chleboun2010finite,rafferty2015monotonicity} and references therein.

As is discussed in more detail in Appendix \ref{appa}, the interpretation of $\rho_b <\rho$ is that the system phase separates into a homogeneous bulk phase and a condensed phase. The latter concentrates on a vanishing volume fraction but contains a non-zero fraction $\rho -\rho_b >0$ of the total mass in the system, and is usually simply called the condensate. Depending on the specific example and the nature of $\pi_{L,N}$ the condensate may cover only a single lattice site (see e.g. \cite{armendariz2009thermodynamic,armendariz2013zero}) or a sub-extensive volume \cite{evans2006interaction,waclaw2009pair}. In most cases the bulk density $\rho_b =\rho_c$ is equal to critical one, but there are also models with $\rho_b <\rho_c$, such as zero-range toy models with size-dependent rates \cite{grosskinsky2008discontinuous,chleboun2015dynamical} which introduce an effective long-range interaction and lead to non-equivalence of ensembles. 
Complete condensation has been established for particular zero-range processes in \cite{jeon2000size, jeon2010phase} and for inclusion processes in a fixed volume in \cite{bianchi2017metastability}. 

As we show in Appendix \ref{appa} in Proposition \ref{propcond}, condensation as defined above implies in particular divergence of higher moments $\langle \eta_x^a \rangle_{L,N}$ with $a>1$. This has been used in some papers as a definition of condensation often using $a=2$ \cite{Evans1998jamming,rajesh2001exact}. The converse does not hold, since moments of limiting distributions $\nu_\rho$ with heavy tails can diverge also in the absence of phase separation, so we stick to Definition \ref{cdef} to characterize condensation. For condensing systems, divergence of higher moments is due to the contribution of diverging occupation numbers in the condensed phase which is not described by the limiting distribution $\nu_\rho$.

\subsection{Models with stationary product measures\label{sec:spm}}

From now on we focus on stochastic particle systems which are defined by a generator of the form
\begin{equation}\label{gene}
\Lcal f(\eta )=\sum_{x,y\in\Lambda} p(x,y) u(\eta_x ,\eta_y )\big( f(\eta^{xy} )-f(\eta )\big)\ ,
\end{equation}
for continuous test functions $f\in C(E_{L,N} )$. 
This defines a continuous-time Markov process on the state space $E_{L,N}$ jumping from configurations $\eta$ to $\eta^{xy}$ where one particle moves from site $x$ to $y$. The spatial dependence of the rates is given by a multiplicative factor $p(x,y)$, which we take to be an irreducible transition kernel for a single particle on $\Lambda$. The interaction between particles is determined by the function $u$ which depends only on the occupation numbers of departure and target site of a jump event. To ensure irreducibility of the process on $E_{L,N}$ we assume
\[
u(m,n)\geq 0\quad\mbox{and}\quad u(m,n)=0\mbox{ if and only if } m=0\ .
\]
To ensure spatial homogeneity at stationarity we assume
\[
\sum_{x\in\Lambda} p(x,y)=\sum_{x\in\Lambda} p(y,x)\quad\mbox{for all }y\in\Lambda\ ,
\]
which is a slight generalization of translation invariance on regular lattices. 
This type of models have first been introduced in the seminal paper \cite{cocozza1985processus}. It is well known (see also \cite{chleboun2014condensation,fajfrova2016invariant}) that they exhibit stationary product measures if and only if
\begin{equation}\label{misa1}
\frac{u(n+1,m)}{u(m+1,n)} = \frac{u(n+1,0)}{u(1,n)}\frac{u(1,m)}{u(m+1,0)} \quad \textrm{for all }\ n,m\geq 0\,,
\end{equation}
and either $p(\cdot,\cdot)$ is symmetric, or
\begin{equation}\label{misa2}
u(n,m) - u(m,n) = u(n,0)-u(m,0) \quad \textrm{for all }\ n,m\geq 0\,.
\end{equation}
In this case, normalizing the weights $\displaystyle w(n) = \prod_{k=1}^n\frac{u(1,k)}{u(k,0)}$ leads to product distributions
\begin{equation}\label{gcm}
\nu_\phi^L [d\eta ]=\Big(\frac{1}{z(\phi )}\Big)^L \prod_{x\in\Lambda} w(\eta_x ) \phi^{\eta_x} \, d\eta\quad\mbox{with}\quad z(\phi )=\sum_{n\in\N_0 } w(n)\,\phi^n\,,
\end{equation}
which are stationary for all $\phi\geq 0$ such that the normalizing partition function $z(\phi )<\infty$. 
Note that these `grand-canonical' distributions are supported on the extended state space $E_L =\big\{ \eta :\eta_x \geq 0\big\}$ without fixing the total number of particles. The expected number of particles per site is given as a monotone increasing function of $\phi$ as
\begin{equation}\label{rhophi}
R(\phi ):=\langle \eta_x \rangle_\phi =\phi\,\partial_\phi\log z(\phi )\ .
\end{equation}
For such processes we have explicit representations of the canonical distributions as conditional grand-canonical distributions
\[
\pi_{L,N} =\nu^L_\phi \Big[\, \cdot\,\Big|\, \sum_{x\in\Lambda} \eta_x =N\Big] ,
\]
which in fact do not depend on the choice of $\phi >0$. This leads to the useful form
\begin{equation}\label{canme}
\pi_{L,N} [d\eta ]=\frac{1}{Z_{L,N}}\prod_{x\in\Lambda} w(\eta_x ) \,\delta \Big( \sum_x \eta_x ,N\Big)\, d\eta\quad\mbox{with}\quad Z_{L,N} =\sum_{\eta\in E_{L,N}} \prod_{x\in\Lambda} w(\eta_x )
\end{equation}
with canonical partition function $Z_{L,N}$. 
This implies in particular that for $\rho <\rho_c$ the limits \eqref{margi} of single-site marginals are given by the marginal $\nu_\phi^1$ with $\phi \geq 0$ such that $R(\phi )=\rho$.

For models of the above type, the condensation transition as given in Definition \ref{cdef} is equivalent to existence of $\phi_c <\infty$ such that $z(\phi )=\infty$ for all $\phi >\phi_c$, and $R(\phi )\to\rho^* <\infty$ as $\phi\to\phi_c$ (see e.g.\ \cite{chleboun2014condensation} for a detailed discussion). Examples of this type studied so far include zero-range processes with $u(m,n)=u(m)$ and decreasing rates $u(m)$ \cite{jeon2000condensation,evans2000phase,armendariz2013zero}., where $\rho^* =\rho_c =\rho_b$. If the rates can depend on the system size, the transition can also be discontinuous with $\rho_b <\rho_c <\rho^*$ where grand-canonical distributions with densities in the range $(\rho_c ,\rho^*)$ are metastable \cite{grosskinsky2008discontinuous,chleboun2015dynamical}. More recently, condensation has also been studied for inclusion processes \cite{grosskinsky2011condensation} and explosive condensation models  \cite{waclaw2012explosive,evans2014condensation,chau2015explosive} with rates of the form
\begin{equation}\label{inrate}
u(m,n)=m^\gamma \big( d +n^\gamma \big)\ ,\quad \gamma\geq 1,\ d>0\ .
\end{equation}
If $\gamma >2$ the system exhibits a condensation transition for all $d>0$ with $\rho_c >0$. 
For inclusion processes we have $\gamma =1$, and this case is covered in more detail in Section \ref{sec:equivalence}. 
In all generic systems with stationary product measures studied so far, we have
$$
\frac{1}{L}\max_{x\in\Lambda} \eta_x \to \rho -\rho_b \quad\mbox{as }L,N\to\infty\ ,\quad N/L\to\rho\ ,
$$
and the condensed phase concentrates on a single lattice site. In Section \ref{sec:main} we will see for the inclusion process that the condensed phase can extend over more than one site and have an interesting hierarchical structure, which has not been observed for condensing particle systems so far.

\subsection{Size-biased sampling\label{sec:sb}}

Since the condensed phase concentrates on a vanishing volume fraction, the limiting marginal probabilities for a fixed number $k$ of occupation numbers converge to the distribution of the bulk in a condensed system. As explained above, for models with stationary product measures this is usually given by the maximal product measure with critical density $\rho_c =R(\phi_c )$ and we have (cf.\ \cite{armendariz2009thermodynamic})
\[
\pi_{L,N} [\eta_{x_1} =n_1 ,\ldots ,\eta_{x_k} =n_k ] \to \prod_{i=1}^k \nu_{\phi_c} (n_i )\ ,
\]
for all $x_1 ,\ldots ,x_k \in\Lambda$ and $n_1 ,\ldots ,n_k \geq 0$. 
This asymptotic equivalence of  canonical and grand canonical ensembles (distributions) has been established for a large class of models \cite{grosskinsky2003condensation,chleboun2014condensation},  and implies weak convergence w.r.t.\ local, bounded test functions as in \eqref{weak}.

Since it contains a non-zero fraction of all particles, the distribution of the condensed phase can be accessed via size-biased permutations of particle configurations. This can be interpreted as picking a particle uniformly at random and sampling the occupation number $\eta_x$ at its location $x$. The larger $\eta_x$, the more likely it is to pick site $x$ in this way. Formally, this can be defined recursively (see e.g.\ \cite{feng2010poisson}, Section 2.4).

\begin{defn}\label{sb}
For given $\eta\in E_{L,N}$ pick a random permutation $\sigma :\Lambda\to\Lambda$ of the lattice indices as
\begin{align*}
\sigma (1)&=x\quad\mbox{with probability}\quad\frac{\eta_x}{N}\ ,\quad x\in\Lambda\ ;\\
\sigma (2)&=x\quad\mbox{with probability}\quad\frac{\eta_x}{N-\eta_{\sigma (1)}}\ ,\quad x\in\Lambda\setminus\{\sigma (1)\}\ ;\\
&\,\,\ldots\quad\mbox{and so on}\ .
\end{align*}
Then we call\quad $\tilde\eta =\big(\tilde\eta_1 ,\ldots ,\tilde\eta_L \big) :=\big( \eta_{\sigma (1)},\ldots ,\eta_{\sigma (L)}\big)$\quad 
a \textbf{size-biased permutation} of $\eta$.
\end{defn}

For models with canonical distributions of the form \eqref{canme}, the distribution of the first size-biased marginal is given by
\begin{equation}\label{smarg}
\pi_{L,N}[\tilde\eta_1 =n] =\frac{L}{N} n\pi_{L,N}[\eta_1 =n]= \frac{L}{N}nw(n)\frac{Z_{L-1,N-n}}{Z_{L,N}}\ ,
\end{equation}
where the stationary weight $w(n)$ is re-weighted proportional to $n$ and re-normalized. 
Here and in the following we use the convention $Z_{L,k} =0$ for all $k<0$, so we can omit indicator functions of the form $\1_{n\leq N}$ to simplify notation. Note that the first identity in \eqref{smarg} with the re-weighted marginal probability holds in general, but the second one only because $\pi_{L,N}$ is a conditional product measure of the form \eqref{canme}.  
For a two-site size-biased marginal we then have 
\begin{align*}
\pi_{L,N}\big[\tilde\eta_1=n_1, \tilde\eta_2=n_2\big]
&=\pi_{L,N}\big[\tilde\eta_2=n_2 \big| \tilde\eta_1=n_1 \big]\,\pi_{L,N}[\tilde\eta_1=n_1]\\
&=\frac{L-1}{N-n_1}\frac{Z_{L-2,N-n_1-n_2}}{Z_{L-1,N-n_1}}n_2w(n_2)\frac{L}{N}\frac{Z_{L-1,N-n_1}}{Z_{L,N}}n_1w(n_1)\\
&=\frac{L(L-1)}{N(N-n_1)}n_1n_2w(n_1)w(n_2)\frac{Z_{L-2,N-n_1-n_2}}{Z_{L,N}}\ .
\end{align*}
Generalizing to the $k$-site case we get
\begin{align}
&\pi_{L,N}\big[\tilde\eta_1=n_1, \tilde\eta_2=n_2,...,\tilde\eta_k=n_k \big]\nonumber\\
&\quad=\frac{L(L-1)\cdots (L-k+1)}{N(N-n_1)\cdots(N-\sum_{i=1}^{k-1} n_i)}\prod_{i=1}^k (n_iw(n_i))\frac{Z_{L-k,N-\sum_{i=1}^k n_i}}{Z_{L,N}},\label{kbimar}
\end{align}
which includes $k=L$ to get the full distribution of $\tilde\eta$ with $Z_{0,n}=1$ for all $n\in\{ 0,\ldots ,N\}$. Note that due to size-biased re-ordering, the distribution of $\tilde\eta$ and its marginals is of course not spatially homogeneous.

To our knowledge, essentially all previous studies of condensation in homogeneous particle systems focus instead on the (decreasing) order statistics
\begin{equation}\label{order}
\hat\eta =\big(\eta_{(1)} ,\ldots ,\eta_{(L)}\big)\quad\mbox{where}\quad \eta_{(1)} \geq \eta_{(2)}\geq \ldots \geq\eta_{(L)}\ ,
\end{equation}
and in particular the maximum occupation number $\eta_{(1)}$ (see e.g.\ \cite{ evans2008extreme,armendariz2009thermodynamic, chleboun2015dynamical,godreche2019condensation}). We will see below how this is related to size-biased sampling, and that the latter is very suitable to study condensation in systems with $\rho_b =0$ such as the inclusion process and related models. A size-biased sampling approach can also be useful in models with $\rho_b >0$ to study the dynamics of the condensed phase and phase separation as recently shown in \cite{jatuviriyapornchai2016coarsening}.

\subsection{The Poisson-Dirichlet and GEM distribution\label{sec:pdgem}}

The Poisson-Dirichlet distribution has been introduced by Kingman in the context of population genetics  \cite{kingman1975random,kingman1977population} and has since occurred in a variety of applications, such as split-merge dynamics \cite{pitman2002,diaconis2004poisson} and random permutations \cite{berestycki2011emergence,betz2011spatial, grosskinsky2012lattice}.
It is a one-parameter family of probability measures defined on the set of ordered partitions of the unit interval
\[
\nabla :=\Big\{(v_1,v_2,...)\in [0,1]^\infty : v_1\geq v_2 \geq \ldots \geq 0, \sum_{j=1}^{\infty} v_j=1\Big\}\ .
\]
It can be characterized for instance as a scaling limit of Dirichlet random variables which form a finite partition of $[0,1]$, or via scale invariant Poisson processes (see Chapter 2 in \cite{feng2010poisson} for details). 
One of the most accessible characterization in terms of practical use is related to the GEM distribution, named in \cite{ewens2004mathematical} after Griffiths \cite{griffiths1980lines,griffiths1988distribution}, Engen \cite{engen2013stochastic} and McCloskey \cite{mccloskey1965model}, which is defined as follows. 
Let $U_1, U_2, \ldots$ be i.i.d. Beta($1,\alpha$) random variables with $\alpha >0$, which take values on $[0,1]$ with PDF $\alpha (1-x)^{\alpha -1}$, and the uniform distribution as a special case for $\alpha =1$. On the set of (unordered) partitions
\[
\Delta :=\Big\{(v_1,v_2,...)\in [0,1]^\infty : \sum_{j=1}^{\infty} v_j=1\Big\}\ .
\]
define a random element $V:=(V_1,V_2,...)\in \Delta$ recursively via
\begin{equation}\label{gem1}
V_1=U_1, V_2=(1-U_1)U_2, V_3=(1-U_1)(1-U_2)U_3, \ldots\ ,
\end{equation}
which corresponds intuitively to breaking off a fraction $1-U_1$ from the unit interval and continuing this process recursively with the remaining interval. The law of $V$ on $\Delta$ is called the \textbf{Griffiths-Engen-McCloskey distribution GEM($\boldsymbol{\alpha}$)}, and the corresponding order statistics $\hat V$ on $\nabla$ has \textbf{Poisson-Dirichlet distribution PD($\boldsymbol{\alpha}$)}. Alternatively, given a PD($\alpha$) distributed partition $V$ on $\nabla$, its size-biased permutation $\tilde V$ has GEM($\alpha$) distribution on $\Delta$ (see e.g. \cite{feng2010poisson} for details).

Note that the construction \eqref{gem1} leads to a hierarchical structure of a GEM($\alpha$) partition $V$, and the paramter $\alpha >0$ controls the expected size of the components. The expectation of Beta$(1,\alpha )$-distributed random variables $U_i$ is $\frac{1}{1+\alpha}$, so for small $\alpha$ the size of the first component $V_1$ is larger and the hierarchy stronger. For larger $\alpha$ the expected sizes of the components are more similar, but always show a strict order since
\begin{equation}\label{gem2}
\Big\langle 1-\sum_{k=1}^n V_k\Big\rangle_{\mathrm{GEM}(\alpha )}=\Big\langle \prod_{k=1}^n(1-U_k)\Big\rangle_{\mathrm{GEM}(\alpha )} =\left(\frac{\alpha}{1+\alpha}\right)^n \to 0\quad\mbox{as }n\to\infty\ .
\end{equation}
This shows that in fact $V\in\Delta$ and that the expected component sizes of $V_k$ vanish as $k\to\infty$, and is also a useful relation to numerically test for GEM distributions (see Section \ref{sec:simu}).

Carrying over the product topology from $[0,1]^\infty$, 
weak convergence of probability distributions on $\Delta$ and $\nabla$ is equivalent to convergence in distribution of finite marginals $(V_1 ,\ldots ,V_k)$ of partitions. 
By Theorem 2 in \cite{donnelly1989continuity}, convergence in distribution of a sequence of size biased partitions $\tilde V^i \to V$ on $\Delta$, implies convergence in distribution of the corresponding ordered partitions $\hat V^i \to \hat V$, and $V$ is a size-biased permutation of $\hat V$. 
In Section \ref{sec:pd} we will use this fact and that rescaled particle configurations $\frac{1}{N}\eta \in\Delta$ can be interpreted as finite partitions of the unit interval, to derive our main results. Note that in a condensing system with $\rho_b <\rho$ \eqref{rhobg} the partitions $\frac{1}{N}\eta$ in the thermodynamic limit only converge on the extended space 
\[
\overline{\Delta} :=\Big\{(v_1,v_2,...)\in [0,1]^\infty : \sum_{j=1}^{\infty} v_j \leq 1\Big\}\ ,
\]
which allows for the loss of mass due to phase separation (see Proposition \ref{propcond} in Appendix \ref{appa}). On the other hand, size-biased permutations capture the condensed phase and the full mass of the system, and $\frac{1}{N}\tilde\eta$ converge on $\Delta$, as we will establish in the next Section.

\section{Condensation in the inclusion process\label{sec:main}}

The inclusion process is a stochastic particle system of type \eqref{gene} with rates
\begin{equation}\label{iprates}
u(m,n)=m(d+n)\quad\mbox{with parameter }d>0\ ,
\end{equation}
which was first introduced in \cite{giardina2007duality} in the context of energy/mass transport. Another important interpretation of this model is as a multi-species version of the Moran model of population genetics, which describes the selection-mutation dynamics of a population of $N$ individuals which can take $L$ different types \cite{mim2017thesis}. Here the parameter $d$ describes the mutation rate, which is small compared to the reproduction rate of the system and is often taken to depend on the system size $d=d_L >0$ and vanish as $L\to\infty$. 
Results in \cite{bianchi2017metastability} show that for fixed $L$ as $N\to\infty$, complete condensation occurs if $d=d_N \ll 1/\log N$. The thermodynamic limit has not been studied so far, and in this section we will establish a complete picture covering all densities $\rho >0$ and possible scaling regimes of the parameter $d$.

The inclusion process satisfies conditions \eqref{misa1} and \eqref{misa2} and has stationary product measures of the form \eqref{gcm} with weights
\begin{equation}\label{weights}
w(n)=\frac{\Gamma (n+d)}{n!\Gamma (d)}\simeq d\, n^{d-1}\quad\mbox{as }n\to\infty\ ,
\end{equation}
\footnote{for functions or sequences we write $f(n)\simeq g(n)$ if $f(n)/g(n)\to 1$ as $n\to\infty$} 
and with normalization $z(\phi )=(1-\phi )^{-d}$. 
So $\phi_c =1$ and
\begin{equation}\label{rphip}
R(\phi )=d\frac{\phi}{1-\phi} \to\infty\quad\mbox{as }\phi\to 1\quad\mbox{for all }d>0\ .
\end{equation}
This also leads to an explicit formula for the canonical distributions
\begin{equation}\label{ipcan}
\pi_{L,N} [d\eta ]=\frac{1}{Z_{L,N}}\prod_{x\in\Lambda} \frac{\Gamma (\eta_x +d)}{\eta_x !\Gamma (d)}\, d\eta\quad\mbox{with}\quad Z_{L,N}=\frac{\Gamma (N+dL)}{N! \Gamma (dL)}\,,
\end{equation}
which can be identified as a Dirichlet multinomial distribution (cf.\ \cite{feng2010poisson}, Chapter 1). These have been studied in detail in the context of urn models and have interesting structural properties and symmetries, but in the following we only make use of the asymptotic form of the partition function so that our results can be more easily translated to other systems. 
Our main results in the thermodynamic limit $N,L\to\infty$, $N/L\to\rho\geq 0$ are derived in the next subsections, and can be summarized as follows:
\begin{enumerate}
\item $d>0$ constant or $d_L \to d>0$: we have asymptotic equivalence of canonical measures and stationary product distributions \eqref{gcm} with $\phi \in [0,1)$ such that $R(\phi )=\rho$ \eqref{rhophi}, and there is no condensation.
\item $d\to 0$: the inclusion process exhibits a condensation transition with $\rho_c =0$ as follows:
\begin{enumerate}
\item $d\to 0$ and $d L\log L\to 0$: complete condensation
\item $d\to 0$ and $d L\to\alpha\in (0,\infty )$: the condensed phase exhibits a hierarchical structure on the scale $N$ given by the PD($\alpha$) distribution.
\item $d\to 0$ and $d L\to\infty$: the condensed phase consists of order $dL$ sites with independent occupation numbers of order $\rho /d$ and exponential distribution.
\end{enumerate}
\end{enumerate}
We will make use of the asymptotic behaviour of $w(n)$ \eqref{weights} and the partition function $Z_{L,N}$, which can be derived by standard Stirling approximations from \eqref{ipcan}.  Particularly useful in the following is the asymptotic behaviour of the ratio
\begin{equation}\label{ratio}
\frac{\Gamma (L+a)}{\Gamma (L+b)} = L^{a-b} \big( 1+o(1)\big)\quad\mbox{as }L\to\infty\ ,
\end{equation}
\footnote{we write $f(n)=o\big( g(n)\big)$ if $f(n)/g(n)\to 0$ as $n\to\infty$} which holds for all sequences $a=a_L$ and $b=b_L$ such that $a^2 ,b^2 \ll L$. Recall also that $\Gamma (d)= \frac{1}{d}\big(1 +o(1)\big)$ as $d\to 0$.

\subsection{Equivalence of ensembles and condensation\label{sec:equivalence}}

We assume $d>0$ constant or $d_L \to d>0$. In this case \eqref{rphip} implies that there exist grand-canonical distributions for any density $\rho\geq 0$, by choosing
\begin{equation}\label{phirho}
\phi =\Phi (\rho ):=\frac{\rho}{d+\rho}\in [0,1)
\end{equation}
such that $R(\phi )=\rho$. In this case the equivalence of ensembles can be established most naturally in terms of the specific relative entropy between canonical and grand-canonical distributions (see e.g.\ \cite{grosskinsky2003condensation,chleboun2014condensation})
\begin{align}
\frac{1}{L}H(\pi_{L,N} ,\nu_\phi^L )&=\frac{1}{L}\sum_{\eta\in E_{L,N}} \pi_{L,N} [\eta ]\log\frac{\pi_{L,N} [\eta ]}{\nu_\phi^L [\eta ]}\nonumber\\
&=\log z(\phi )-\frac{N}{L}\log\phi -\frac{1}{L}\log Z_{L,N}\ .\label{relen}
\end{align}
Computing the leading order terms of $Z_{L,N}$ from \eqref{ipcan} with standard Stirling formula we get
$$
\frac{1}{L}\log Z_{L,N} \to \rho\log\Big( 1+\frac{d}{\rho}\Big) +d\log\rho\ ,
$$
so choosing $\phi =\phi (\rho )$ as in \eqref{phirho} we see that \eqref{relen} vanishes in the thermodynamic limit since $\log z(\phi )=-d\log (1-\phi )$. 
Convergence in specific relative entropy implies convergence of finite marginals \cite{chleboun2014condensation}, i.e. for any fixed $k>0$ and $n_1 ,\ldots ,n_k \geq 0$
\begin{align*}
\pi_{L,N} \big[ \eta_1 =n_1 ,\ldots ,\eta_k =n_k\big]\to \frac{1}{z(\phi )^k}\prod_{i=1}^k w( n_i ) \phi (\rho )^{n_i}\quad\mbox{as }N/L\to\rho\ .
\end{align*}
The latter limit could also be computed directly in analogy to other results below, but the route via the equivalence of ensembles is more robust since only the logarithm of the partition function has to be controlled to leading order.

An alternative representation of the specific relative entropy is given by (see e.g. \cite{grosskinsky2003condensation})
\begin{align*}
\frac{1}{L}H(\pi_{L,N} ,\nu_\phi^L )&=-\frac{1}{L}\log \nu_{\phi}^L\Big[\sum_{x \in \Lambda}\eta_x = N\Big]\, .
\end{align*}
Since the second moment of the single-site marginal $\nu_\phi$ is finite when $\phi (\rho )=\rho /(\rho +d)<1$, one can show that this vanishes in the thermodynamic limit even without computing the asymptotics of $Z_{L,N}$, by applying a local central limit theorem to the right hand side 
(see for example \cite{chleboun2011large,davis1995elementary}).

In the case $d\to 0$, \eqref{rphip} implies that there are no grand-canonical distributions for any positive density and therefore we expect a condensation transition, following the discussion after \eqref{canme}. We summarize this in the following result proved by a direct computation.

\begin{prop}\label{conprop}
Provided that $d\to 0$ as $L\to\infty$, the inclusion process exhibits a condensation transition as given in Definition \ref{cdef} with $\rho_c =\rho_b =0$, i.e. we have for all fixed $n\geq 0$ and $\rho\geq 0$
$$
\pi_{L,N} [\eta_1 =n] \to\delta_{0,n}\quad\mbox{as }L,N\to\infty ,\ N/L\to\rho >0\ .
$$
\end{prop}

\begin{proof}
We have for any $n\geq 0$ fixed
\begin{equation}
\pi_{L,N} \big[ \eta_1 =n \big] = w(n)\frac{Z_{L-1,N-n}}{Z_{L,N}}\simeq w(n)\ ,
\end{equation}
since with the scaling \eqref{zscaling} given below for the partition function in the case $dL\to\alpha\in [0,\infty )$ we have
$$
\frac{Z_{L-1,N-n}}{Z_{L,N}} \simeq \Big( 1-\frac{n}{N}\Big)^{dL} \to 1\ .
$$
The same holds with \eqref{zratio} in the case $dL\to\infty$. 
From \eqref{weights} we have $w(0)=1$ and $w(n)=O(d)$ for any $n\geq 1$, leading to $\pi_{L,N} [\eta_1 =n] \to\delta_{0,n}$, independently of $\rho$. 
With Definition \ref{cdef} this implies condensation with $\rho_c =\rho_b =0$.
\end{proof}

So locally the system appears empty in the limit, and a further investigation of the condensed phase will be given below in terms of size-biased samples. Note that in the proof we only use the asymptotic behaviour of ratios of partition functions and the fact that $w(n)=O(d)$ for all $n>0$.

\subsection{GEM scaling limit and complete condensation\label{sec:pd}}

We study the distribution of the condensed phase by computing size-biased marginals in the case $dL\to \alpha \geq 0$. Using \eqref{ratio}, the leading order behaviour of the partition function is given by
\begin{equation}\label{zscaling}
Z_{L,N} =\frac{\Gamma (N+dL)}{N!\Gamma (dL)}\simeq \left\{\begin{array}{cl}
 d N^{dL} /\rho &\mbox{ if }dL\to 0\,,\\ N^{dL -1}/\Gamma (\alpha )&\mbox{ if }dL\to\alpha >0\,.
\end{array}\right.
\end{equation}
Recall from Section \ref{sec:pdgem} that $\frac{1}{N} (\eta_1 ,\ldots ,\eta_L )$ is a (finite) partition of the unit interval.

\begin{thm}\label{thm1}
In the thermodynamic limit $L,N\to \infty$ such that $N/L \to \rho$ with $dL\to\alpha>0$, the rescaled order statistics of $\eta$ \eqref{order} converge in distribution to Poisson Dirichlet, i.e.
\begin{equation}\label{thm1e}
\frac{1}{N}\hat\eta =\frac{1}{N} \big( \eta_{(1)} ,\ldots ,\eta_{(L)}\big)\stackrel{D}{\longrightarrow} \mathrm{PD}(\alpha )\ .
\end{equation}
Equivalently, size-biased samples converge as\quad $\frac{1}{N}\tilde\eta \stackrel{D}{\longrightarrow}\mathrm{GEM}(\alpha )$\ .
\end{thm}

\begin{proof}
Following the discussion in Section \ref{sec:pdgem} it suffices to show that for all $k\geq 1$, $x_1 ,\ldots ,x_k \in [0,1]$ we have
\begin{equation}\label{toshow}
N(N-n_1)\cdots\Big( N-\sum_{i=1}^{k-1} n_i\Big)\pi_{L,N}[\tilde\eta_1=n_1, \tilde\eta_2=n_2,...,\tilde\eta_k=n_k]\to \alpha^k\prod_{i=1}^k (1-x_i)^{\alpha-1},
\end{equation}
provided that $\frac{n_1}{N}\to x_1\in[0,1] , \frac{n_2}{N}\to (1-x_1)x_2, \cdots,  \frac{n_k}{N}\to (1-x_1)(1-x_2)\cdots(1-x_{k-1})x_k$. With the characterization in \eqref{gem1} this establishes convergence in distribution of size-biased permutations to GEM($\alpha$), which is equivalent to \eqref{thm1e}. 

Using \eqref{kbimar}, the scaling of $w(n)\simeq dn^{d-1}$  as $n\to\infty$ \eqref{weights} and the partition function \eqref{zscaling}, and \eqref{ratio} we get
\begin{align}\label{pdcomp}
&\pi_{L,N}[\tilde\eta_1=n_1, \tilde\eta_2=n_2,...,\tilde\eta_k=n_k]\nonumber\\
&\quad=\frac{L(L-1)\cdots (L-k+1)}{N(N-n_1)\cdots(N-\sum_{i=1}^{k-1} n_i)}\frac{Z_{L-k,N-\sum_{i=1}^{k} n_i}}{Z_{L,N}}\prod_{i=1}^k (n_iw(n_i))\nonumber\\
&\quad\simeq\frac{L(L-1)\cdots (L-k+1)d^k}{N(N-n_1)\cdots(N-\sum_{i=1}^{k-1} n_i)} \Big(\frac{N-\sum_{i=1}^{k} n_i}{N}\Big)^{dL -1} \Big(N-\sum_{i=1}^k n_i\Big)^{-dk} \prod_{i=1}^k n_i^d\ .
\end{align}
Since $d=O(1/L)$ we have $n_i^d \to 1$ and also $\big(N-\sum_{i=1}^k n_i\big)^{-dk} \to 1$. Furthermore, with the choice of $n_i$ we have
$$
1-\frac{1}{N}\sum_{i=1}^k n_i \simeq (1-x_1 )\cdots (1-x_k )
$$
which implies \eqref{toshow}.
\end{proof}
For $\alpha \to 0$ the above limiting distribution PD($\alpha$) degenerates, with the mass fraction of the maximal occupation number tending to $1$. Under a mild additional assumption $dL\ll 1/\log L$ on the scaling, this statement can be significantly strengthened to ensure complete condensation in analogy with results in \cite{bianchi2017metastability} for fixed $L$ as $N\to\infty$.

\begin{prop}\label{propc}
In the thermodynamic limit $L,N\to \infty$ such that $N/L \to \rho$ with $dL\log L\to 0$, we have complete condensation in the sense of \eqref{complete}, i.e. $\pi_{L,N} \big[\max_{x\in\Lambda} \eta_x=N\big]\to 1$.
\end{prop}

\begin{proof}
It suffices to show for the first size-biased marginal that
\begin{equation}\label{toshow2}
\pi_{L,N} [\tilde\eta_1 =N-n] \to \delta_{n,0}\quad\mbox{for all }n\geq 0\ ,
\end{equation}
which implies the same for the maximal occupation number. Using again \eqref{smarg}, \eqref{weights} and \eqref{zscaling} we have for all $n\geq 0$ 
\begin{align*}
\pi_{L,N} [\tilde\eta_1 =N-n]
&= \frac{L}{N}(N-n)w(N-n)\frac{Z_{L-1,n}}{Z_{L,N}} \simeq \frac{d}{\rho }(N-n)^{d} \frac{Z_{L-1,n}}{d N^{dL}/\rho}\\
&= \Big( 1-\frac{n}{N}\Big)^{d} N^{-d(L-1)} Z_{L-1,n}\ .
\end{align*}
The first term tends to $1$ for all $n\geq 0$ and the second scales like
$$
N^{-d(L-1)} =e^{-d(L-1)\log N}\to 1\quad\mbox{since }dL\ll 1/\log L\ .
$$
Then $Z_{L-1 ,0}=1$ and $Z_{L-1 ,n}\simeq dL /n\to 0$ for $n\geq 1$, which implies \eqref{toshow2}.
\end{proof}

\subsection{Intermediate scales\label{sec:inter}}

Assuming that $d\to 0$ with $dL\to\infty$ we cannot easily apply \eqref{ratio} for asymptotic estimates, and after a slightly more involved Stirling approximation the leading order of the partition function \eqref{canme} is
\begin{equation}\label{ziscale}
Z_{L,N}\simeq \frac{e^{-1}}{\sqrt{2\pi dL}}\Big(\frac{N}{dL}\Big)^{dL-1} \Big( 1+\frac{dL}{N}\Big)^{N+dL}\ .
\end{equation}
While in principle this scaling together with that of the weights \eqref{weights} fully determines the asymptotics of size-biased distributions, it turns out to be more useful to use particular cancellations when estimating ratios of partition functions to proof our main result below. The above scaling implies for all fixed $n\geq 0$ that
\begin{equation}\label{zratio}
\frac{Z_{L-1,N-n}}{Z_{L,N}}\simeq \big( 1-n/N\big)^{dL} \big( 1+1/L\big)^{dL} \big( 1+dL/N\big)^{-n} \to 1\ ,
\end{equation}
which we have used to prove Proposition \ref{conprop}.

\begin{thm}\label{interthm}
In the thermodynamic limit $L,N\to \infty$ such that $N/L \to \rho$, $d\to 0$ and $dL\to\infty$, we have for any $\rho> 0$ and fixed $k\in\N$
$$
d(\tilde\eta_1 ,\ldots ,\tilde\eta_k )\stackrel{D}{\longrightarrow}\mbox{i.i.d. Exp}(1/\rho )\ .
$$
i.e. marginals of rescaled size-biased samples $\tilde\eta$ converge in distribution to independent exponential random variables with mean $\rho$.
\end{thm}

\begin{proof}
To establish convergence of the joint density we have to show for all $n_1 ,\ldots ,n_k$ such that $n_i d\to x_i >0$
\begin{equation}\label{toshownow}
\frac{1}{d^k} \pi_{L,N} \big[ \tilde\eta_1 =n_1,\ldots ,\tilde\eta_k =n_k ]\to \frac{1}{\rho^k} \exp\Big( -\sum_{i=1}^k x_i\Big/\rho\Big)\ .
\end{equation}
In an analogous computation to \eqref{pdcomp}, we get
\begin{align}
&\frac{1}{d^k}\pi_{L,N} [\tilde\eta_1=n_1, \tilde\eta_2=n_2,\ldots ,\tilde\eta_k=n_k]\nonumber\\
&\quad =\frac{1}{d^k}\frac{L(L-1)\cdots (L-k+1)}{N(N-n_1)\cdots(N-\sum_{i=1}^kn_i)}\frac{Z_{L-k,N-\sum_i n_i}}{Z_{L,N}}\prod_{i=1}^k (n_iw(n_i))\nonumber\\
&\quad\simeq \Big(\frac{1}{\rho}\Big)^k \underbrace{\prod_{i=1}^k \Big(\frac{x_i}{d}\Big)^d }_{:=A}\ \underbrace{\frac{\Gamma (N-\sum_i n_i +d(L-k))}{(N-\sum_i n_i )!}}_{:=B}\underbrace{\frac{N!}{\Gamma (N+dL)}}_{:=C}\underbrace{\frac{\Gamma (dL)}{\Gamma (d(L-k))}}_{:=D}\ ,
\label{expo1}
\end{align}
where we used the asymptotic behaviour of the stationary weights \eqref{weights}, and arranged the contributions of the ratio of partition functions in a convenient way. Since $d\to 0$ we have $A\to 1$ and $D\simeq (dL)^{dk}$ using \eqref{ratio}. The latter does not apply to the other two terms since $dL\to\infty$, and a more careful (but straightforward) analysis leads to
$$
C\simeq N^{1-dL} \Big( 1+\frac{d}{\rho}\Big)^{N+dL} \Big( 1-\frac{1}{N}\Big)^N e^{dL-1}
$$
and analogously, using $\frac{\Gamma (N-\sum_i n_i +d(L-k))}{\Gamma (N-\sum_i n_i+dL)}\simeq \big( N-\sum_i n_i +dL\big)^{-kd}\simeq N^{-kd}$,
$$
B\simeq N^{-kd}\big( N-\sum_i n_i \big)^{dL-1} \bigg(\Big( 1+\frac{d}{\rho}\Big) \Big( 1+\frac{d\sum_i n_i}{\rho N}\Big)\bigg)^{\sum_i n_i -N-dL} \Big( 1-\frac{1}{N-\sum_i n_i}\Big)^{\sum_i n_i -N} e^{1-dL}\ .
$$
Therefore we get
$$
BCD\simeq (d/\rho )^{dk} \Big( 1+\frac{d}{\rho}\Big)^{\sum_i x_i /d} \Big( 1-\frac{\sum_i x_i}{\rho dL}\Big)^{dL} \Big( 1+\frac{\sum_i x_i}{\rho N}\Big)^{-N}\to e^{-\sum_i x_i /\rho}\ , 
$$
and inserting into \eqref{expo1} implies \eqref{toshownow}.
\end{proof}

So the condensed phase for any intermediate scale with $dL\to\infty$ has a non-hierarchical structure, locally consisting of independent clusters of average size $\rho /d$. This general behaviour across a large range of scaling regimes is quite remarkable. However, since $dN\to\infty$, the rescaled size-biased samples $d\tilde\eta$ do not form a partition of a compact interval (as in the previous case of $dL\to\alpha$). So our result on convergence of finite marginals does not imply weak convergence of the full sequence $d\tilde\eta$, and we only get a local characterization of the condensed phase. Since the total mass of the condensed phase is $N$, and $k$ in the above result can be chosen arbitrarily large, this at least implies that the volume fraction covered by the condensed phase scales at least as $d$ to leading order.

Note also that the limiting exponential distribution of a rescaled cluster in the condensed phase is not itself the size-biased distribution of a random variable, since this would have density
\[
\frac{\rho}{x} \frac{1}{\rho}e^{-x/\rho } =\frac{1}{x}e^{-x/\rho }\ .
\]
This cannot be normalized due to divergence at $x=0$, and  suggests that the condensed phase does not simply consist of $O(1/d)$ clusters with i.i.d. occupation numbers. If, conditional on the volume covered by the condensed phase, one could probe a cluster size without size bias, it would vanish on the scale $1/d$. This suggests that the volume fraction covered by the condensed phase could indeed be larger than $d$ with many clusters on smaller scales that do not contribute to the total mass to leading order. Details of this behaviour are most likely depending on the particular scaling of $d$, and are very hard to access analytically or even to observe numerically.

\subsection{Simulation results\label{sec:simu}}

\begin{figure}
\begin{center}
\includegraphics[width=0.45\textwidth]{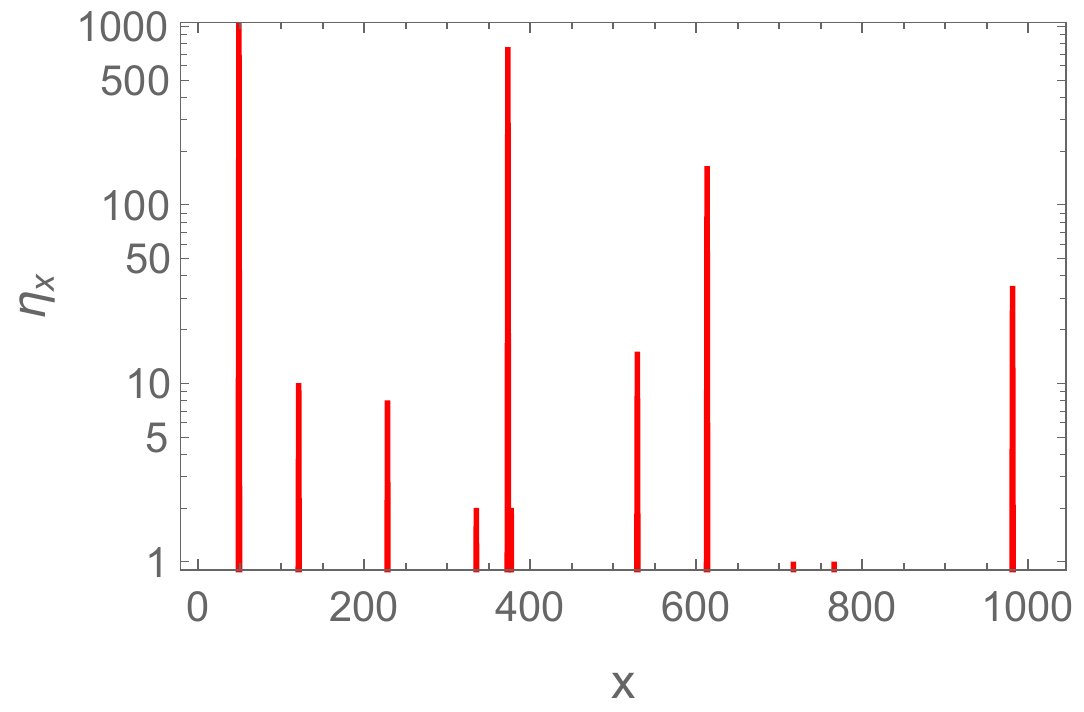}\quad\includegraphics[width=0.45\textwidth]{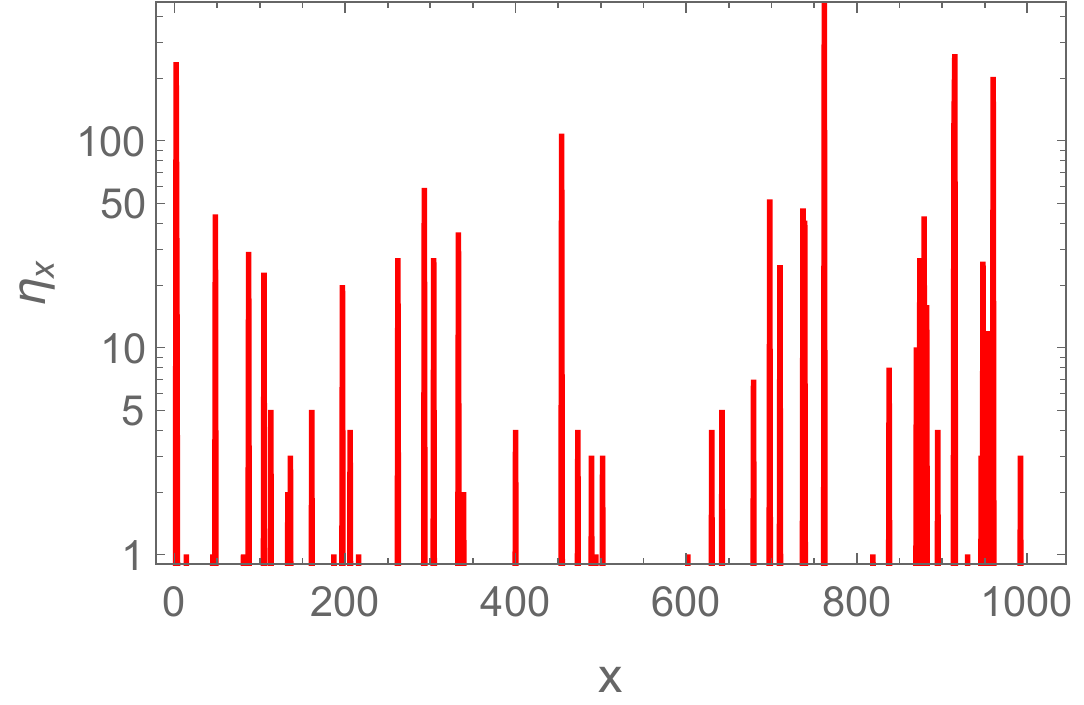}
\end{center}
\caption{\label{fig:confs}
Typical stationary configurations for the inclusion process with $N=2048$ particles on a lattice of size $L=1024$ for TA dynamics with $dL=1$ (left) and CG dynamics with $dL=10$ (right).}
\end{figure}

We illustrate our main results with Monte Carlo simulations of the inclusion process at stationarity. Recall that with \eqref{gene} and \eqref{iprates} the generator describing the dynamics is given by
\begin{equation}\label{ipgen}
\Lcal f(\eta )=\sum_{x,y\in\Lambda} p(x,y) \eta_x (d+\eta_y )\big( f(\eta^{xy})-f(\eta )\big)\ .
\end{equation}
We initialize the system by distributing $N$ particles independently, uniformly at random on the lattice. The stationary distributions $\pi_{L,N}$ \eqref{ipcan} are conditional product measures for all translation invariant or symmetric choices of $p(x,y)$. On the complete graph with $p(x,y)\equiv\frac{1}{L-1}$ one can implement a simple rejection based algorithm to simulate the dynamics, which we summarize in Appendix \ref{appb} and call CG dynamics in the following. We also implemented the standard Gillespie algorithm \cite{gillespie1976general} to simulate totally asymmetric dynamics on a one-dimensional lattice with periodic boundary conditions, i.e. $p(x,y)=\delta_{y,x+1\mathrm{mod}L}$, which we call TA dynamics. 

In both geometries, the number of empty sites grows in time and the particles concentrate in clusters, which exchange particles. Smaller clusters disappear and the average cluster size increases, driving a coarsening process. This leads to stationary distributions where either a balance between cluster aggregation and break-up is reached, which is the case for $d\to 0$ and $dL\to\alpha \in (0,\infty ]$, or the system saturates with a single cluster remaining for $dL\to 0$. 
While for CG dynamics clusters can directly exchange particles, for TA dynamics the clusters are isolated and the coarsening process is limited by particle transport, which has been studied in \cite{cao2014dynamics}. Still, once stationarity is reached (see Appendix \ref{appb} for more details on this), both dynamics provide samples from the same stationary distributions $\pi_{L,N}$ which do not have any spatial correlations. Two typical stationary configurations for CG and TA dynamics are illustrated in Figure \ref{fig:confs}.

\begin{figure}
\begin{center}
\includegraphics[width=0.45\textwidth]{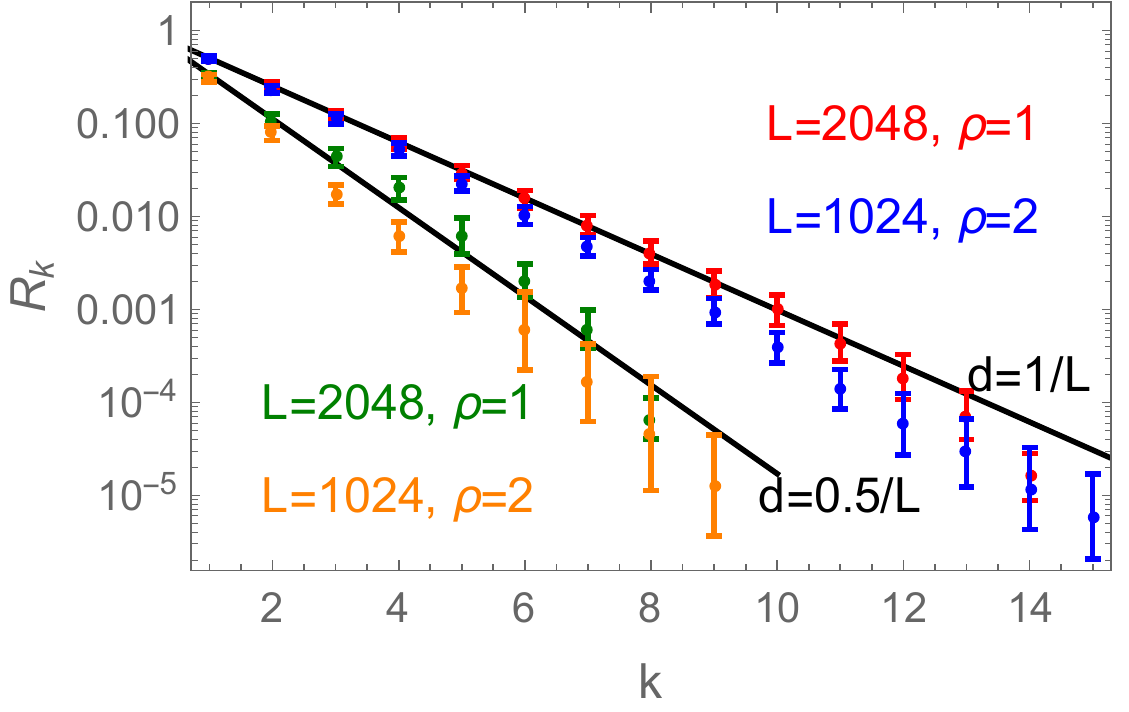}\quad\includegraphics[width=0.45\textwidth]{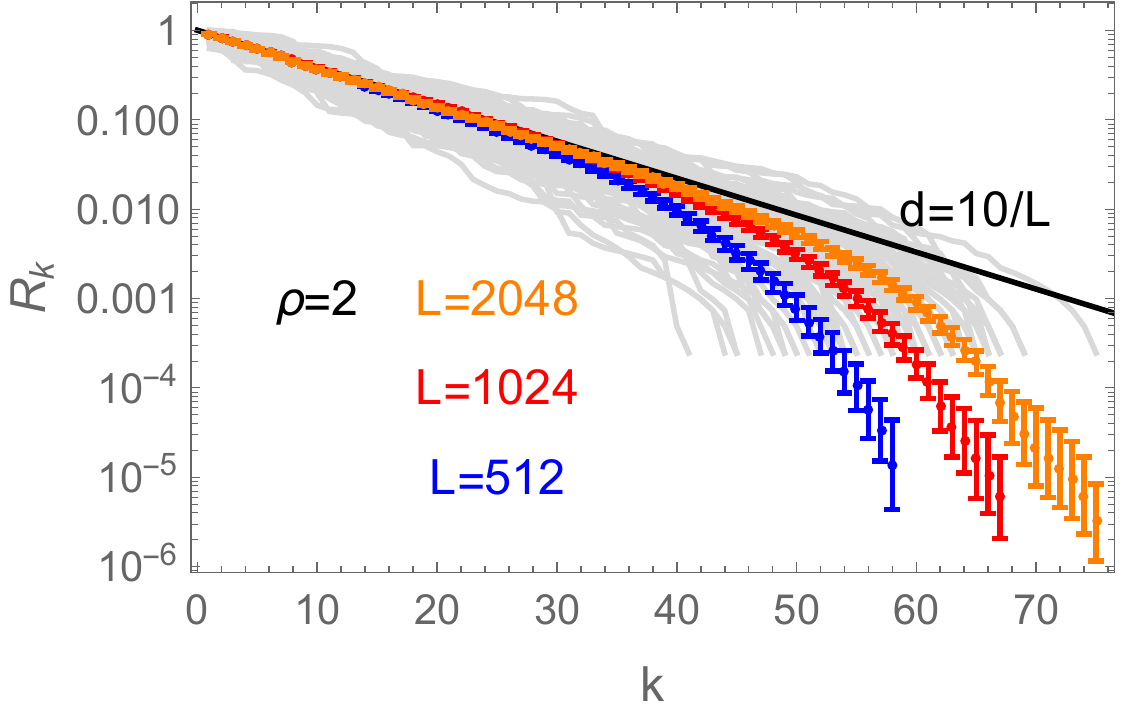}
\end{center}
\caption{\label{fig:gem}
Sample averages of $R_k$ \eqref{rk} against $k$ from CG dynamics for the inclusion process are compared to the expected limiting behaviour \eqref{erk} (black lines) for $dL=0.5$ and $1$ (left) and $dL=10$ (right). Data are given in coloured symbols with error bars and averaged over $100$ realizations $\eta$ and a further $5$ size-biased re-samples $\tilde\eta$ for each. Grey lines on the right show $100$ individual $R_k (\tilde\eta )$ for $L=2048$, the smallest possible non-zero value of $R_k$ here is $1/4096$.}
\end{figure}

Since the complete condensation regime $dL\to 0$ has been studied numerically before \cite{cao2014dynamics}, we focus on the hierarchical results in Theorem \ref{thm1} with $dL\to\alpha\in (0,\infty )$, and comment on intermediate scales with $dL\to\infty$ from Theorem \ref{interthm} later. 
There are no particularly useful results for marginals of Poisson Dirichlet random variables, so we compare size-biased samples of stationary configurations $\tilde\eta$ to the GEM($\alpha$) distribution. For each $k\geq 1$, we define
\begin{equation}\label{rk}
R_k (\tilde\eta):=1-\frac{1}{N}\sum_{i=1}^k \tilde\eta_i \ ,
\end{equation}
the mass fraction remaining on all sites with index $>k$ in the size-biased sample $\tilde\eta$. With the representation \eqref{gem1} of the GEM distribution, Theorem \ref{thm1} implies that for each $k\geq 1$ the random variable $R_k$ converges in distribution to a product of i.i.d.\ random variables $1-U_i$, where $U_i \sim \mathrm{Beta}(1,\alpha)$. With \eqref{gem2} this implies that
\begin{equation}\label{erk}
\langle R_k \rangle_{L,N} \to \Big(\frac{\alpha}{1+\alpha }\Big)^k \quad\mbox{as }L,N\to\infty,\ N/L\to\rho,\ dL\to\alpha\ ,
\end{equation}
which is illustrated in Figure \ref{fig:gem} for various values of $\alpha$ and $\rho$. We see good agreement for small values of $k$, but in addition to statistical errors there are large systematic finite-size effects (illustrated for $\alpha =10$ in Fig.\ \ref{fig:gem} right). These are related to the small amount of non-zero occupation numbers $\# (\eta )$ in typical stationary configurations, leading to a systematic underestimation of $\langle R_k \rangle_{L,N}$. This can be derived from Ewen's sampling formula (see e.g.\ \cite{feng2010poisson}, Theorem 2.8), where $\# (\eta )$ corresponds to the number of different types in a finite sample of size $N$ from a Poisson-Dirichlet population, and can be shown to scale as
\[
\# (\eta )\simeq \alpha\log N\quad\mbox{as }L,N\to\infty,\ N/L\to\rho,\ dL\to\alpha\ .
\]
This logarithmic scaling can be seen in Figure \ref{fig:gem} (right). Convergence of $\# (\eta )/\log N$ to $\alpha$ is very slow on the scale $1/\sqrt{\log N}$ (see \cite{feng2010poisson}, Theorem 2.11), so this is not a good estimator for $\alpha$, and the comparison based on \eqref{erk} in Figure \ref{fig:gem} is more useful.

\begin{figure}
\begin{center}
\includegraphics[width=0.45\textwidth]{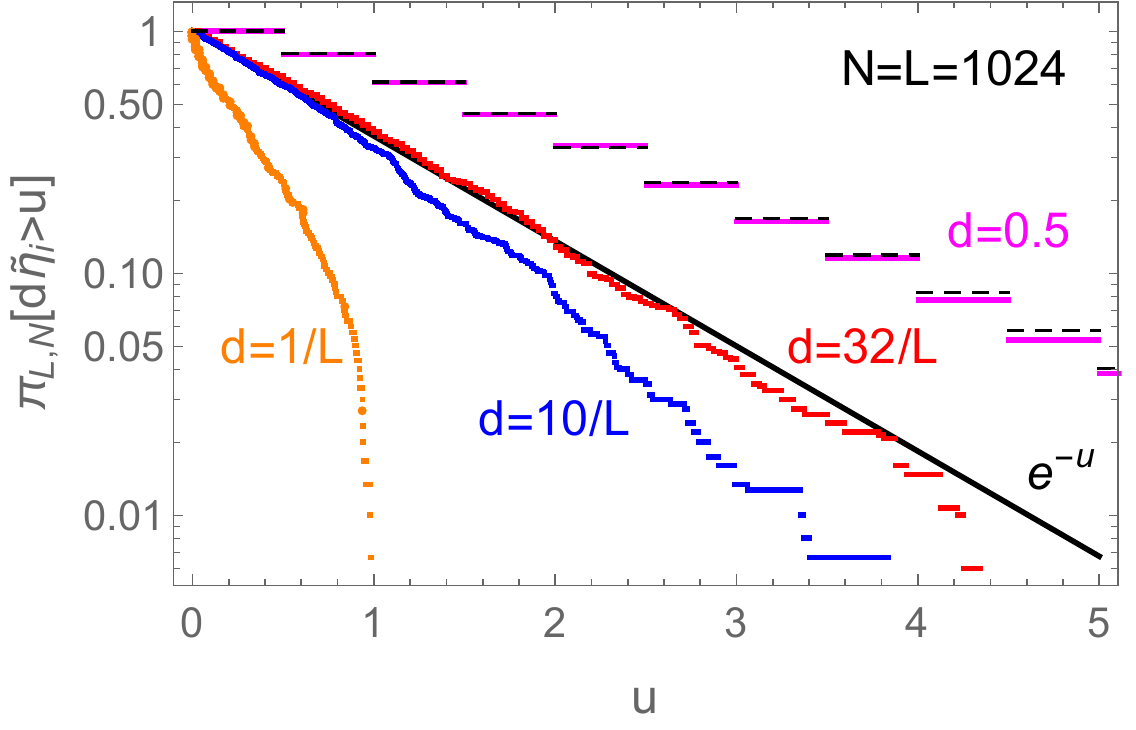}\quad\includegraphics[width=0.45\textwidth]{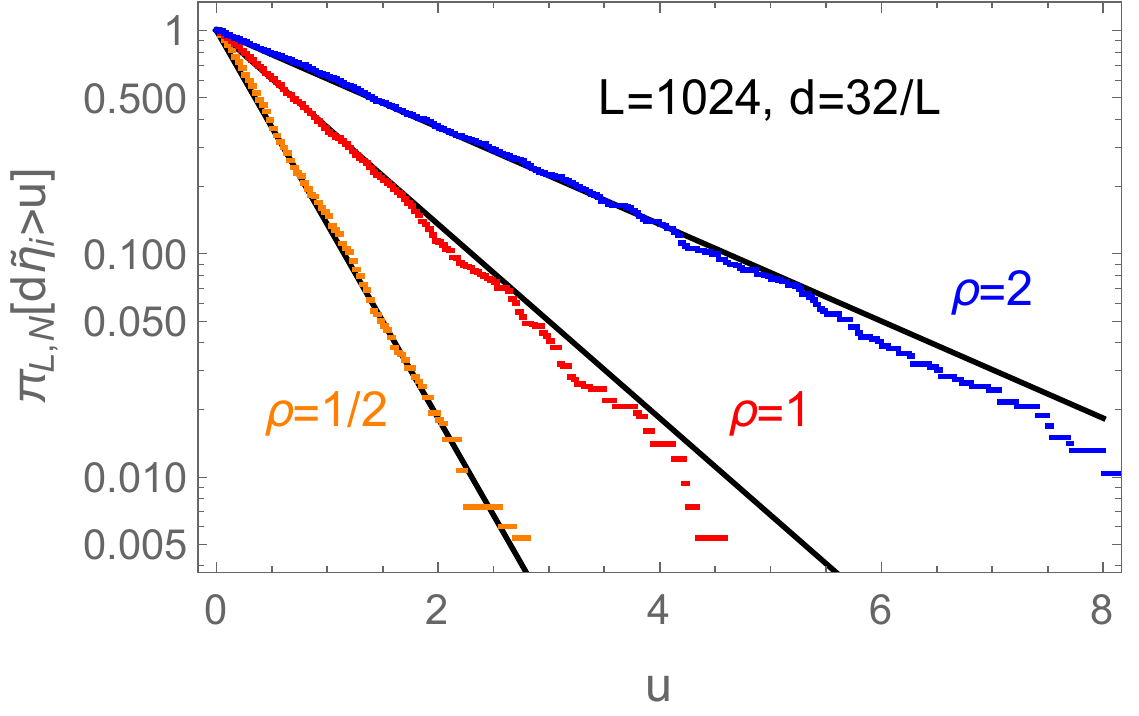}
\end{center}
\caption{\label{fig:inter}
Empirical tail distributions of $d\tilde\eta_i$ for $i=1,2,3$ from $100$ samples $\eta$ and $5$ size-biased re-samples $\tilde\eta$ are shown as coloured step functions, and compared to the theoretical prediction $e^{-u/\rho}$ from Theorem \ref{interthm} for the intermediate regime (full black lines). On the left we fix $\rho =1$ and agreement with $e^{-u}$ improves with increasing $d$.  We also include the size-biased grand canonical prediction \eqref{predid} for the regime of constant $d>0$ (dashed black line), which agrees well with the discrete data for $d=0.5$. On the right we fix $d=32=1/\sqrt{L}$, with good agreement with theory for densities $\rho =0.5,\, 1$ and $2$.}
\end{figure}

For small values of $d$ and finite system size $L$ there is a data cross-over to the condensed regime, with very few occupied sites.  This is very hard to access numerically, but theoretically, a single condensate site is fully consistent with the limit $\alpha\to 0$ in \eqref{erk}. 
For large values of $d$ there is a data cross-over to the intermediate regime $d\to 0$ with $dL\to\infty$, which is covered by Theorem \ref{interthm}. This cross-over is illustrated in Figure \ref{fig:inter} (left), where we plot the empirical tail distribution of $d\tilde\eta_i$ for $i=1,2,3$ based on $5$ size-biased re-samples $\tilde\eta$ of $100$ independent samples of $\eta$ from $\pi_{L,N}$ using CG dynamics. We pick small values for $i$ in order to use the same procedure for all values of $d$ including $1/L$. For larger $d$, larger values for $i$ lead to the same behaviour, and tests reveal that the samples $\tilde\eta_i$ are indeed uncorrelated. For fixed density $\rho =1$ we see that agreement with the exponential tail, $e^{-u/\rho}$ predicted by Theorem \ref{interthm}, improves with increasing $d$ up to $d=32/L=1/\sqrt{L}$. In Figure \ref{fig:inter} (right) for this value of $d$ we see good agreement with the predicted tail for several densities $\rho$.

If we increase $d$ further the system crosses over to the behaviour for constant $d>0$, where we have equivalence of ensembles to grand canonical measures $\nu_\phi$ as explained in Section \ref{sec:equivalence}. Rescaled size-biased variables $d\tilde\eta_i$ will then take discrete values in $d\N$ given by the size-biased version of $\nu_\phi^1$ \eqref{gcm}, i.e.
\begin{equation}\label{predid}
\pi_{L,N} \big[ d\tilde\eta_i =dn\big] \to \frac{n}{\rho}\nu_{\phi (\rho )}^1 [\eta_x =n]=\frac{nw(n)}{\rho z(\phi (\rho ))} \phi (\rho )^n
\end{equation}
as $L,N\to\infty$, $N/L\to\rho$ and $d>0$ fixed. Here $\phi (\rho )=\rho /(d+\rho )<1$ is given in \eqref{phirho} and $z(\phi )=(1-\phi )^{-d}$. This is illustrated for $d=512 L=0.5$ in Figure \ref{fig:inter} (left), where we compare the empirical tail with the tail of the size-biased distribution \eqref{predid} and see very good agreement. Note that for $d\to 0$, we have from the right-hand side of \eqref{predid} that
\[
\frac{1}{d} \frac{n}{\rho}\nu_{\phi (\rho )}^1 [\eta_x =n]\to \frac{1}{\rho} e^{-u/\rho } \quad\mbox{if }nd\to u\ ,
\]
since $nw(n)/d\to 1$, $z(\phi (\rho ))\to 1$ and $\phi (\rho )^n \to e^{-u/\rho }$. So the size-biased grand-canonical distributions scale consistently with the result in Theorem \ref{interthm}.

\section{Large deviations\label{sec:ld}}

In Section \ref{sec:main} we derived the typical stationary behaviour in the condensed phase, and will now study the statistics of large deviations of the maximum occupation number. The most interesting case of complete condensation is covered in Section \ref{sec:ldcom}, for completeness and to introduce the main concepts of large deviations we first cover the non-condensing and intermediate regime. Note that in the hierarchical regime with $dL\to\alpha\in (0,\infty )$, the typical size of the maximum is of order $L$ and it can take any value on that scale with non-vanishing probability.

\subsection{Non-condensing regime}
We first treat the case $d \to d > 0$ as $L \to \infty$ for which we have equivalence of ensembles. 
We find that the probability of observing maximum site occupations of order $L$ decays exponentially in $L$, as would be the case under the grand-canonical measures $\nu_\phi$ \eqref{gcm} where the site occupations are i.i.d.\ with finite mean and variance.
We characterise this decay in terms of the large deviation rate function $I_{\rho}(m)$, which is informally defined as
\[
\pi_{L,N}[\eta_{(1)} = M] \sim e^{-L I_{\rho}(m)},\quad \textrm{for}\quad L,N,M\to \infty\ \textrm{ and }\ N/L \to \rho,\ M/L \to m\,.
\]
This is made precise in the following result which characterizes the local large deviations, and provides an explicit form for the rate function. 
The results in this section imply large deviation principles in the usual sense, see for example \cite{touchette2009large,chleboun2011large} and references therein for details.

\begin{prop}\label{prop:fluidLDP}
	If $d \to d>0$ and $m \in [0,\rho)$, then in the thermodynamic limit
	\begin{align}
	\frac{1}{L}\log \pi_{L,N} \big[\eta_{(1)}=M]\to -I_\rho(m)\quad \textrm{as} \quad N/L\to\rho\ ,\quad M/L \to m \in [0,\rho)\,, 
	\end{align}
	where
	\begin{align}
		\label{eq:fluidLDP}
	I_\rho(m) = (\rho -m) \log  \frac{\rho-m}{\rho-m+d} -\rho \log \frac{\rho}{\rho +d} - d \log \frac{\rho-m+d}{\rho +d}\,.
	\end{align}
\end{prop}
\begin{proof}
The proof follows a standard tilting argument which we only sketch here, more details can be found in \cite{chleboun2011large}. 
First note that for grand-canonical measures \eqref{gcm} with $\phi ,\phi' \in [0,1)$
\begin{equation}\label{tilt}
\nu_{\phi}^L\Big[\sum_x \eta_x = N\Big] =\nu_{\phi'}^L\Big[\sum_x \eta_x = N\Big] \Big(\frac{\phi}{\phi'}\Big)^N \Big(\frac{z(\phi')}{z(\phi )}\Big)^L \ ,
\end{equation}
and recall that $\nu_\phi^1 [\eta_1 =n]=w(n)\phi^n /z(\phi )$ with weights $w(n)$ given in \eqref{weights} and normalization $z(\phi )=(1-\phi )^{-d}$ for all $\phi\in [0,1 )$. 
Since
\[
\pi_{L,N}[\eta_{(1)} = M] = \nu_{\phi}^L\Big[\eta_{(1)} =M \big|\sum_x \eta_x = N\Big]\ ,
\]
and $(\eta_x :x\in \Lambda )$ are i.i.d. under $\nu_\phi^L$, we have
\begin{align*}
  \frac{1}{L}\log \pi_{L,N}&[\eta_{(1)} = M] = \frac{1}{L} \log \nu_{\phi}^L\Big[\eta_{(1)} = M;\ \sum_x \eta_x = N\Big] - \frac{1}{L}\log\nu_{\phi}^L\Big[\sum_x \eta_x = N\Big] \\
   &=  \frac{1}{L} \log \nu_\phi^{L-1}\Big[\sum_{x}\eta_x = N-M ;\ \eta_{(1)} \leq M\Big] + \frac{1}{L}\log \nu_\phi[\eta_1 = M]\\
  &\qquad  
   - \frac{1}{L}\log\nu_{\phi}^L\Big[\sum_x \eta_x = N\Big]\ . 
\end{align*}
Since the grand canonical single site marginals $\nu_\phi$ have finite exponential moments for each $\phi \in [0,1 )$, we may choose a sequence of $\phi$ such that the expected number of particles per site under $\nu_{\phi}[\,\cdot\, ;\, \eta_1 < M]$ is $(N-M)/(L-1)$. Further, since $M/L\to m$, this implies $\phi \to \Phi(\rho -m)$ in the thermodynamic limit, with $\Phi$ given in \eqref{phirho} as the inverse of $R(\phi )$ \eqref{rphip}. 
Since $\nu_{\phi}$ has second moment which converges to $\langle \eta_x^2\rangle_{\Phi(\rho-m)} < \infty$, we may then apply a standard local limit theorem for triangular arrays (see e.g. \cite{davis1995elementary}) to show that with this choice of $\phi$ the first term on the second line vanishes. The same is true for the term in the third line choosing $\phi =\Phi (\rho )= \rho/(\rho+d)$ by equivalence of ensembles proved in Section \ref{sec:equivalence}, and we can conclude using \eqref{tilt} and taking limits.
\end{proof}

\subsection{Intermediate scales}
For the intermediate scale, $d \to 0$ with $dL \to \infty$, we cannot directly apply a local limit theorem for triangular arrays as in the previous case, since with \eqref{rphip} there are no grand-canonical measures with positive densities. 
Here we will make use of Stirling's approximation of the partition function \eqref{ziscale} and truncation arguments to derive the large deviations behaviour of the maximum $\eta_{(1)}$. 
In this regime the probability of observing a maximum site occupation of order $L$ has asymptotic decay rate $dL$. 
\begin{prop}\label{prop:interLDP}
	If $d \to0$ and $dL \gg \log L$, then in the thermodynamic limit we have 
	\begin{align}
	\label{eq:interLDP}
	-\frac{1}{d L}\log \pi_{L,N} \big[ \eta_{(1)}=M\big]\to I_\rho(m):=\log\left(\frac{\rho}{\rho-m}\right) \,,
	\end{align}
	as $N/L\to\rho$ and $M/L\to m\in [0,\rho)$.
\end{prop}
Note that this rate function is consistent with the limit $d\to 0$ of $I_\rho (m)/d$ in \eqref{eq:fluidLDP}, but the case $d=0$ is not covered by Proposition \ref{prop:fluidLDP} and needs a separate proof.

\begin{proof}
	We firstly extract the contribution due to the maximum site occupation by observing that
        \begin{equation}\label{pbounds}
          \frac{w(M)Z_{L-1,N-M}^{(M)}}{Z_{L,N}}\leq	\pi_{L,N} \big[\eta_{(1)}=M\big] \leq \frac{L w(M)Z_{L-1,N-M}^{(M)}}{Z_{L,N}}\,,
        \end{equation}
        where\quad $\displaystyle Z_{L,N}^{(M)} = \sum_{\eta \in E_{L,N}}\prod_{x \in \Lambda} w(\eta_x)\mathbbm{1}\{\eta_x \leq M\}$\quad
is a truncated canonical partition function.\\
This immediately implies the upper bound
\[
\pi_{L,N} \big[\max_{x\in\Lambda}\eta_x=M\big]
\leq \frac{Lw(M) Z_{L-1,N-M} }{Z_{L,N}}\,.
\]
We can bound from above the total weight of configurations violating the truncation by
\[
Z_{L-1,N-M} -Z_{L-1,N-M}^{(M)}\leq (L-1)(N-M)w(M) Z_{L-2,N-2M}\ ,
\]
where we use monotone decay in $N$ of the weights $w(N)$ \eqref{weights} and the partition function $Z_{L,N}$ \eqref{canme}, which holds since $dL > 1$ for $L$ sufficiently large. 
This leads to a lower bound on $Z_{L-1,N-M}^{(M)}$ in \eqref{pbounds} and we get
\[
\pi_{L,N} \big[\eta_{(1)} =M\big] 
\geq\frac{w(M)Z_{L-1,N-M}}{Z_{L,N}}\left(1 - (L-1)(N-M)w(M)\frac{ Z_{L-2,N-2M}}{Z_{L-1,N-M}}\right)\,.
\]
By applying \eqref{ziscale} together with \eqref{weights} we find that 
\[
(L-1)(N-M)w(M)\frac{ Z_{L-2,N-2M}}{Z_{L,N-M}} \to 0\,, 
\]
        in the thermodynamic limit if $M/L \to m > 0$.
        We conclude by taking logarithms, and again applying \eqref{ziscale} together with \eqref{weights}.  
\end{proof}

We illustrate the rate function for this and the following case of complete condensation in Figure \ref{fig:LDP} and compare to exact numerics obtained for finite system size. The latter are generated using the right-hand side of \eqref{pbounds} and the recursive structure of the canonical partition functions
\begin{equation}\label{numerics}
Z_{L,N} =\sum_{n=0}^N Z_{k,n} Z_{L-k,N-n} \quad\mbox{for all }k=1,\ldots L-1\ .
\end{equation}
The same relation holds for truncated partition functions (see \cite{chleboun2011large} for details). With initial condition $Z_{1,n} =w(n)$, $n=0,\ldots N$ and choosing $k=L/2$ this can be used effectively in an iteration to reach large system sizes.

\subsection{Complete condensation\label{sec:ldcom}}

\begin{figure}
\begin{center}
\includegraphics[width=0.48\textwidth]{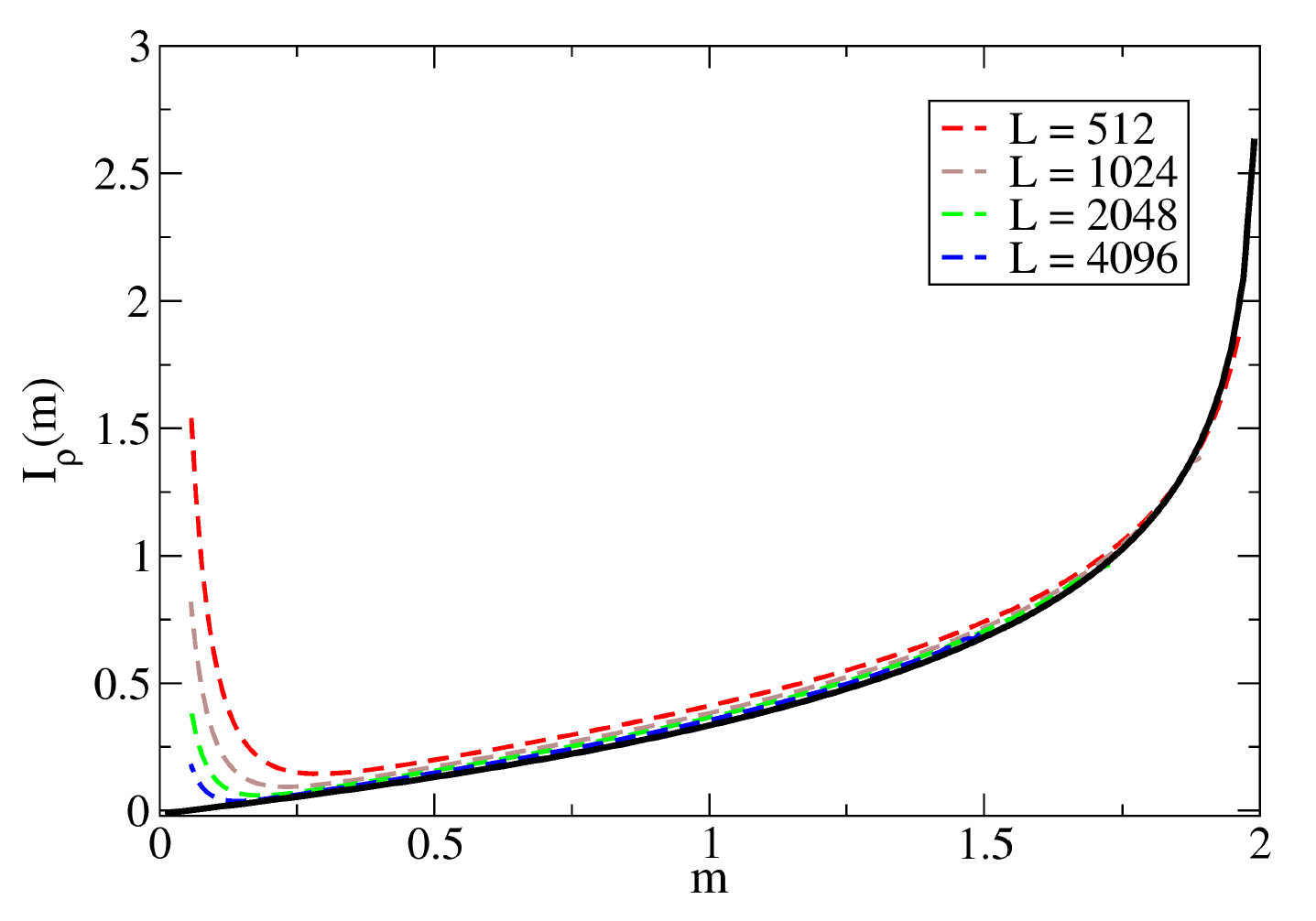}\quad\ \includegraphics[width=0.48\textwidth]{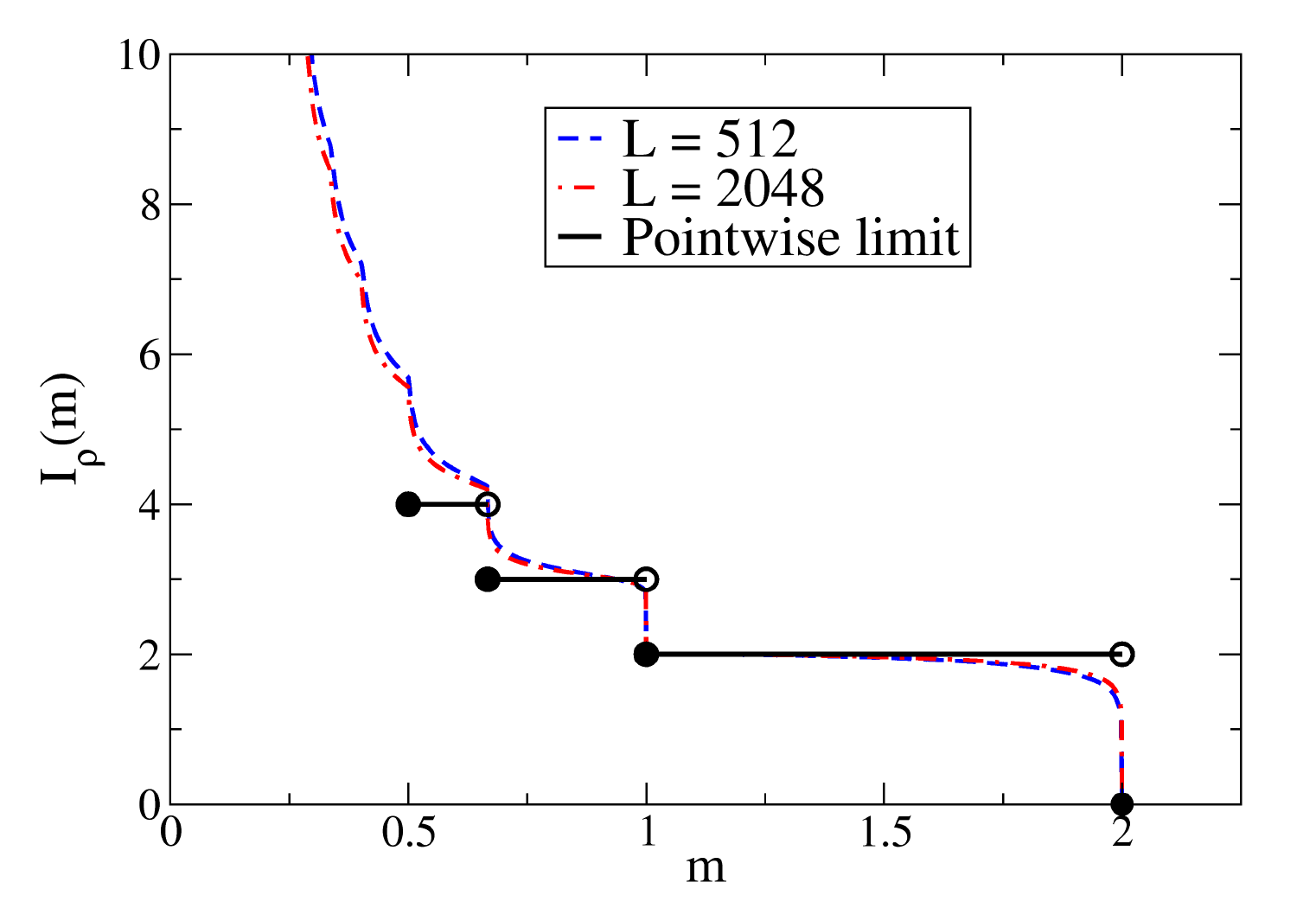}
\end{center}
\caption{\label{fig:LDP}
The large deviation rate functions of the maximum site occupation, $I_\rho(m)$. Theoretical results are given by full black lines and numerics \eqref{numerics} for finite $L$ by dashed coloured lines. Left: The intermediate case as given in \eqref{eq:interLDP}, with numerics for $d=1/\sqrt{L}$. According to Theorem \ref{interthm} the maximum typically contains of order $1/d =\sqrt{L}$ particles, so the location of the minima of $I_\rho (m)$ vanishes with $1/\sqrt{L}$ and there are significant finite size effects close to the origin. Right: The complete condensation case as given in \eqref{eq:ratefntotal}, with numerics for $d=L^{-2}$.\\
}
\end{figure}

In the case $d L \ll 1/\log L$ we have complete condensation as stated in Proposition \ref{propc}.
We characterise the large deviations of the maximum on the scale $L$, which turn out to be dominated by the probability of observing the smallest number of occupied sites required to realise a given size of the maximum. 
To derive this result, it is easier to first understand probabilities of size-biased configurations in analogy to Theorem \ref{thm1}.

\begin{prop}\label{prop:condLDP}
	In the thermodynamic limit $N,L\to \infty$ such that $N/L \to \rho$, with $dL \log L\to 0$ we have 
	\begin{equation}
	\label{eq:condLDP}
	\frac{1}{d^k}\pi_{L,N}[\tilde\eta_1=n_1,...,\tilde\eta_k=n_k] \to \rho^{-k}\prod_{i=1}^{k}(1-x_i)^{i-k-1}\,,
	\end{equation}
	provided that $\frac{n_1}{N}\to x_1 , \frac{n_2}{N}\to (1-x_1)x_2, \cdots \frac{n_k}{N}\to (1-x_1)(1-x_2)\cdots(1-x_{k-1})x_k$ with $x_1,x_2,\ldots,x_k \in (0,1)$. Furthermore, in the same limit
	\begin{equation}
	\label{eq:condLDP-equiv}
	\pi_{L,N}[\tilde\eta_1=n_1,...,\tilde\eta_k=n_k] \simeq \pi_{L,N}\left[\tilde\eta_1=n_1,...,\tilde\eta_k=n_k,\tilde\eta_{k+1}=N-\sum_{i=1}^k n_i\right].
	\end{equation}
\end{prop}
\begin{proof}
In analogy to \eqref{pdcomp} in the proof of Theorem \ref{thm1} we get
	\begin{equation}
	\frac{1}{d^k}\pi_{L,N}[\tilde\eta_1=n_1,...,\tilde\eta_k=n_k] \simeq \rho^{-k}\prod_{i=1}^{k-1}(1-x_i)^{i-k}\frac{Z_{L-k,N-\sum_{i=1}^k n_i}}{Z_{L,N}}\,,
	\end{equation}
where we also used
\[
N(N-n_1)\cdots(N-\sum_{i=1}^{k-1} n_i)\simeq L^k \rho^k \prod_{i=1}^{k-1} (1-x_i )^{k-i}\ .
\]
	The remaining mass, $N-\sum_{i=1}^k n_i$, is of order $L$ since  $x_1,x_2,\ldots,x_k \in (0,1)$.
	Therefore, applying \eqref{zscaling} to the ratio of partition functions we find
\[
		\frac{Z_{L-k,N-\sum_{i=1}^k n_i}}{Z_{L,N}} \to \frac{1}{(1-x_1)(1-x_2)\ldots(1-x_k)}\,,
\]
	where we used $N^{dL}$, $(N-\sum_{i=1}^k n_i)^{dL} \to 1$, since $dL \log L \to 0$.
	This completes the proof of \eqref{eq:condLDP}.
	
	Finally, for \eqref{eq:condLDP-equiv}, we let $n_{k+1}=N-\sum_{i=1}^kn_i$. Then using \eqref{kbimar} and the fact that $Z_{L,0}=1$ for all $L\geq 1$ we have
		\begin{align*}
	\frac{\pi_{L,N}\left[\tilde\eta_1=n_1,...,\tilde\eta_k=n_k,\tilde\eta_{k+1}=n_{k+1}\right]}{\pi_{L,N}[\tilde\eta_1=n_1,...,\tilde\eta_k=n_k]} &= \frac{(L-k)n_{k+1}w(n_{k+1})Z_{L-k-1,0}}{n_{k+1} Z_{L-k,n_{k+1}}}\\
		&=\pi_{L-k,n_{k+1}}[\eta_{(1)} = n_{k+1}]\,,
	\end{align*} 
	which tends to one by Proposition \ref{propc}. 
\end{proof}

\begin{cor}\label{prop:condLDP-max}
	In the thermodynamic limit $N,L\to \infty$ such that $N/L \to \rho$, $M/N \to x\in (0,1)$ with $dL \log L\to 0$ we have 
	\begin{equation}
	\label{eq:condLDPmax}
	\frac{1}{d^{\lceil 1/x \rceil-1}N^{\lceil 1/x \rceil-2}}\pi_{L,N} \big[\eta_{(1)}=M\big]\to \frac{C(x)}{\rho^{\lceil 1/x\rceil -1}}\,,
	\end{equation}
        where $0<C(x)<\infty $ is an $x$ dependent constant.
\end{cor}
\begin{proof}
	The result follows rather directly from the previous proposition, and we sketch the main calculations required.
	First fix $M \in [N/2,N)\cap \mathbb{N}$, then conditioned on the event  $\{\eta_{(1)}{=} M\}$ the configuration must contain at least two non-empty sites. Observe that $\{\eta_{(1)}{=} M\}$ is given by the disjoint union
	\[ \{\eta_{(1)} = M\} = \{\eta_{(1)}=M,\,\tilde\eta_3 > 0\} \cup \{\tilde\eta_{1} = M,\tilde\eta_{2}=N-M\}\cup\{\tilde\eta_{1} = N-M,\tilde\eta_{2}=M\} \,.
	\] 
	From \eqref{eq:condLDP-equiv} in Proposition \ref{prop:condLDP} we see that $\pi_{L,N} \big[\eta_{(1)}=M\,;\,\tilde\eta_3 > 0 \big]$ decays to zero faster than $d$.
	Applying \eqref{eq:condLDP} to the probability of the remaining two events we find
	\[
		\frac{1}{d} \pi_{L,N} \big[\eta_{(1)}=M\big]  \to \frac{1}{\rho x (1-x)}\quad\mbox{so\quad $C(x)=\frac{1}{x(1-x)}$ for $x\in [1/2,1)$}\ .
	\]

More generally, fix $k \in \mathbb{N}$ and $M \in [N/(k+1),N/k)$, let $n_1 = M$, then we can again decompose as a disjoint union as follows
	\begin{align*}
		 \{\eta_{(1)} = M\} =& \{\eta_{(1)}=n_1,\,\tilde\eta_{k+2} > 0 \big\}\\
		 & \bigcup_{\sigma \in S_{k+1}}\bigcup_{\substack{n_2,\ldots,n_{k}\,:\,\\ n_1\geq n_2\geq\ldots\geq n_{k+1}}} \{ \tilde \eta_{1}=n_{\sigma(1)},\tilde \eta_{2}=n_{\sigma(2)},\ldots,\tilde \eta_{k+1}=n_{\sigma(k+1)}\}
	\end{align*}
	where $S_{k+1}$ is the set of permutations of $\{1,2,\ldots,k+1\}$ and $n_{k+1} =N-\sum_{i=1}^{k}n_{i}$. In order for $n_1\geq n_2\geq\ldots\geq n_{k+1}$ to hold we must have that $(k+1-i)n_{i+1} \geq N-\sum_{j=1}^{i} n_j$ for each $i \in \{1,\ldots,k-1\}$.
	 Again with \eqref{eq:condLDP-equiv}, the probability of the  event $\{\eta_{(1)}=n_1,\,\tilde\eta_{k+2} > 0 \big\}$ decays faster than $d^{k}L^{k-1}$. Applying \eqref{eq:condLDP} yields
	\[
	\frac{1}{d^k}\pi_{L,N}[\eta_{(1)} = M]\simeq N^{k-1} \frac{1}{\rho^k}\underbrace{\sum_{\sigma \in S_{k+1}} \int^x_{\frac{1-x}{k}}{\dots}\int_{\frac{\prod_{i=1}^{k-1}(1-x_i)}{2}}^{x_{k-1}} \prod_{i=1}^{k}(1-x_{\sigma(i)})^{i-k-1} dx_k{\ldots} dx_{2}}_{:=C(x)}\,, 
	\]
        and \eqref{eq:condLDPmax} follows.
\end{proof}

If we take $d = L^{-\gamma}$ with $\gamma > 1$ then we may summarize Corollary \ref{prop:condLDP-max} in terms of a large deviation rate function (with speed $\log L$), as follows
\begin{align}
  \label{eq:ratefntotal}
  -\frac{1}{\log L} \log \pi_{L,N}[\eta_{(1)} = M] \to I_\rho(m) = (\lceil \rho/m \rceil - 1) \gamma - (\lceil \rho/m \rceil - 2).
\end{align}
This is illustrated in Figure \ref{fig:LDP} (right) for $\gamma = 2$.

\section{Discussion\label{sec:dis}}

\subsection{Summary}

We have established a complete picture for condensation in the inclusion process in the thermodynamic limit, and characterized the condensed phase in several regimes using size-biased sampling of configurations. Our results cover the full scaling regime of the diffusion parameter $d$, only excluding some narrow bands of size $\log L/L$ for complete condensation and large deviations. A particularly interesting regime is the hierarchical structure discussed in Section \ref{sec:pd} related to the GEM and the Poisson-Dirichlet distribution. This is well established in the context of population genetics \cite{feng2010poisson}, where the full structure of Dirichlet multinomials has been exploited to derive very detailed results for Moran models, which can be interpreted as inclusion processes. We derived our results using only the most general properties of inclusion processes so that our approach can be easily transferred to other systems, and we give more details in the next subsection.

The Poisson-Dirichlet distribution has been identified as the unique stationary distribution of split-merge dynamics of clusters \cite{pitman2002,diaconis2004poisson}, where split and merge rates are proportional to cluster sizes. Our results show that the inclusion process can be seen as a generic 'monomer exchange' version of such dynamics, where now only single particles are exchanged but with the same proportionality of rates in the inclusion interaction. 
It would be very interesting to investigate this connection in detail in the context of Poisson-Dirichlet diffusions in analogy to \cite{costantini2017wrightfisher}. 
The crucial prerequisite to see Poisson-Dirichlet statistics in particle systems such as the inclusion process is the asymptotic behaviour of the stationary weights \eqref{weights},
\[
w(0)=1\ ,\quad w(n)=O(d) \mbox{ for all }n\geq 1\mbox{ as }L\to\infty\ ,\quad\mbox{and}\quad w(n)/d\simeq n^{-1} \mbox{ as }n\to\infty\ .
\]
The fact that $w(n)$ vanishes proportionally to $d$ as $L\to\infty$ for all $n>0$ leads to $\rho_b =\rho_c =0$ and condensation with an empty bulk. The structure of the condensed phase is determined by the $1/n$ decay of stationary weights for large occupation numbers. This is quite robust, as is discussed in the next subsection. There we summarize some previous results and connections to other particle systems with Poisson-Dirichlet statistics.

\subsection{Other particle systems with Poisson-Dirichlet statistics}

The model studied in \cite{jack2017emergence} consists of $N$ particles moving diffusively on a one-dimensional torus of length $L$, subject to a logarithmic attractive potential and short-range hard-core exclusion. The weak attraction leads to the formation of large gaps between groups of particles, and the distances $y=(y_1 ,\ldots ,y_N )$ between particles have a stationary distribution of the form \eqref{canme} with weights $w(y)=y^{-\beta}$, 
where $\beta <1$ corresponds to a dimensionless inverse temperature controlling the strength of the noise. So the rescaled distances $\frac{1}{L} y$ provide a partition of the unit interval and follow a Dirichlet($1-\beta ,\ldots ,1-\beta$) distribution. Of particular interest in \cite{jack2017emergence} is the temperature scaling $\beta=\frac{N-b}{N-1}\nearrow 1$ as $N\to \infty$ with $b>1$, where Theorem 2.1 in \cite{feng2010poisson} directly applies so that the order statistics
\[
\frac{1}{L}\hat y \stackrel{D}{\longrightarrow} \mathrm{PD}(b-1 )\quad\mbox{as }N,L\to\infty\ ,\quad L/N\to\rho\ ,
\]
converges in distribution to a Poisson-Dirichlet partition of $[0,1]$. Indeed, the corresponding Beta($1,b-1$)  distribution of the first size-biased marginal $\tilde y_1$  as in \eqref{gem1} is established independently in \cite{jack2017emergence} without mentioning the connection to the Poisson-Dirichlet distribution. 
Note that in this model gaps between particles correspond to cluster sizes, and the average cluster size is therefore $L/N$. 
A related paper with a hierarchical clustering phenomenon for interacting diffusions on a ring is \cite{andres2010particle}, and to our knowledge these continuous models are the only particle systems where a connection to Poisson-Dirichlet statistics has been recognized so far. The Brownian energy process introduced in \cite{giardina2007duality,giardina2009duality} as a dual model to the inclusion process exhibits stationary product measures with chi-squared marginals, and conditioning on the total sum of occupation numbers leads to the same canonical distributions as the model in \cite{jack2017emergence}.

To test the robustness of our results against small changes in the stationary weights $w(n)$, it is useful to consider zero-range processes. For any given $w(n)$ it is well known that a process with the jump rate for a cluster of size $n$ to lose a particle given by
\[
u(n) =\frac{w(n-1)}{w(n)}\quad\mbox{for }n\geq 1
\]
exhibits stationary product measures of the form \eqref{gcm}  (see e.g.\ \cite{spitzer1970interaction,cocozza1985processus} and references therein). 
Using the weights \eqref{weights} for the inclusion process this leads to jump rates
\begin{equation}\label{zrates}
u(n)=\frac{n}{d+n-1}\quad\mbox{for all }n\geq 1\ ,
\end{equation}
so that $u(1)=1/d$ diverges in a scaling limit with $d\to 0$. All other rates are bounded and converge as
\[
u(n)\to \frac{n}{n-1} \quad\mbox{as }L\to\infty\mbox{ for all }n\geq 2.
\]
A zero-range process with rates \eqref{zrates} has exactly the same stationary distributions \eqref{canme} as the inclustion process and all our results apply. Condensation in zero-range processes has been a major research area in recent years (see e.g.\  \cite{evans2000phase,grosskinsky2003condensation, chleboun2014condensation}), where decreasing rates $u(n)\simeq 1+b/n$ lead to stationary weights of order $n^{-b}$, so that $\phi_c =1$ and the critical density is given by (see discussion in Section \ref{sec:spm})
\[
\rho_c =R(1) =\frac{1}{z(1)}\sum_{n=1}^\infty nw(n)<\infty \quad\mbox{for }b>2\ .
\]
In such models, condensation is driven by strong enough on-site attraction between particles. The rates \eqref{zrates} have asymptotic behaviour
\begin{equation}\label{zrates2}
u(n)\simeq \frac{n}{n-1} \simeq 1+\frac{1}{n}\quad\mbox{as }n\to\infty
\end{equation}
and the attraction between particles is not strong enough. Instead, cluster coarsening and condensation is driven by divergence of $u(1)=1/d$, which ensures that $\rho_b =0$ in the bulk of the system and the remaining mass concentrates on a number of lattice sites decreasing in time.

We have checked numerically that the particular form of the rates \eqref{zrates} is in fact not important, and choices of the form $u(n)=n/(n-1)$ or $u(n)=1+1/n$ for $n\geq 2$ lead to the expected Poisson-Dirichlet statistics at stationarity for $u(1)=1/d\simeq L/\alpha$ with $\alpha >0$. 
This can be checked analytically on a case-by-case basis, but it is known that in general the asymptotic behaviour of the partition function and condensation behaviour may depend sensitively on perturbations of the rates (see e.g.\ \cite{jeon2010phase} and  \cite{delmolino2012condensation,grosskinsky2008instability}), so we are currently not able to prove a general result analogous to Theorem \ref{thm1} based only on asymptotics of stationary weights or jump rates.

\section*{Acknowledgements}
We are grateful to Robert Jack for helpful discussions and comments. S.\ G.\ acknowledges partial support from the Engineering and Physical Sciences Research Council (EPSRC), Grant No.\ EP/M003620/1. 
This research project is supported by Mahidol University.

\appendix

\section*{Appendix}
\section{Condensation and phase separation\label{appa}}

For completeness we summarize some implications of Definition \ref{cdef} on phase separation and divergence of higher moments, using only the definition itself without any further assumptions on the canonical measures.  
Assume that we have a condensing particle system on the state space $E_{L,N}$ according to Definition \ref{cdef}, with canonical distributions $\pi_{L,N}$ and limiting single-site marginal $\nu_\rho$ as defined in \eqref{margi}. 
Weak convergence of $\pi_{L,N}$ to $\nu_\rho$ in the thermodynamic limit $N,L\to\infty$, $N/L\to\rho$ is equivalent to convergence of expectations of bounded test functions, so that for any $K>0$
$$
\langle \eta_x \1_{\eta_x \leq K}\rangle_{L,N} \to \langle \eta_x \1_{\eta_x \leq K}\rangle_\rho\ .
$$
Now taking a second limit $K\to\infty$ the right-hand side converges to $\rho_b =\langle\eta_x \rangle_\rho$, which is strictly smaller than $\rho$ in a condensing system (so that both limits do not commute).

The two limits in this order can be used to characterize phase separation as explained in Section \ref{sec:setting} on the level of single-site marginals, where $\eta_x \1_{\eta_x \leq K}$ describes the bulk part of the distribution and $\eta_x \1_{\eta_x > K}$ the condensed part. Definition \ref{cdef} implies that the condensed phase is supported on a vanishing volume fraction but contains a non-zero fraction of the total mass. In the limit $L,N\to\infty$, $N/L\to\rho$ and then $K\to\infty$ we get
\begin{alignat}{4}
& && &&\quad\mbox{condensed} && \quad\mbox{bulk/background}\nonumber\\
&\ \ \mbox{mass fraction} &&\quad \langle\eta_x \rangle_{L,N} &&=\langle\eta_x \1_{\eta_x >K} \rangle_{L,N}\ &&+\ \langle \eta_x \1_{\eta_x \leq K}\rangle_{L,N}\nonumber\\
& &&\quad\ \ \to\rho &&\quad\to\rho -\rho_b &&\qquad\to\rho_b\nonumber\\
&\mbox{volume fraction} &&\quad \langle 1\rangle_{L,N} &&=\langle\1_{\eta_x >K} \rangle_{L,N}\ &&+\ \langle \1_{\eta_x \leq K}\rangle_{L,N}\nonumber\\
& &&\quad\ \ =1 &&\quad\to 0 &&\qquad\to 1\qquad .\label{phasesep}
\end{alignat}
This follows simply from convergence for bounded test functions in the bulk and conservation of total probability and mass. It implies in particular that in this ordered limit
$$
\langle \eta_x |\eta_x \leq K\rangle_{L,N} \to \rho_b \quad\mbox{and}\quad \langle \eta_x |\eta_x > K\rangle_{L,N} \to \infty
$$
for the average occupation numbers in the bulk and condensed phase, respectively.

A further interesting property that is often used is that condensation leads to the divergence of higher order moments, due to the contribution of the condensed phase. This is implied by the following general result.

\begin{prop}\label{propcond}
Assume that a system exhibits condensation as in Definition \ref{cdef} in the thermodynamic limit with density $\rho$. Then for all $x\in\Lambda$ and any positive function $f:\N_0 \to\R^+$ with $f(n)\to\infty$ as $n\to\infty$ we have 
\begin{equation}
\big\langle \eta_x f(\eta_x ) \big\rangle_{L,N} \to\infty \quad\mbox{and}\quad \big\langle \eta_x /f(\eta_x ) \big\rangle_{L,N} \to\big\langle \eta_x /f(\eta_x ) \big\rangle_{\rho}\,,
\end{equation}
as $L,N\to\infty$, and $N/L\to\rho$.
\end{prop}

\begin{proof}
For any fixed $K>0$ we have
\begin{align*}
\big\langle \eta_x f(\eta_x ) \big\rangle_{L,N} &=\sum_{n=0}^\infty n f(n)\pi_{L,N} [\eta_x =n] \geq \min_{n>K} f(n) \sum_{n=K+1}^\infty n \pi_{L,N} [\eta_x =n]\\
&=\min_{n>K} f(n)\Big(\frac{N}{L} -\langle\eta_x \1_{\eta_x \leq K}\rangle_{L,N}\Big) \to \min_{n>K} f(n)\big(\rho -\langle\eta_x \1_{\eta_x \leq K}\rangle_\rho \big)
\end{align*}
as $L,N\to\infty$, $N/L\to\rho$. 
This holds for all $K>0$ and $\rho -\langle\eta_x \1_{\eta_x \leq K}\rangle_\rho \to \rho -\rho_b >0$ as $K\to\infty$ with \eqref{phasesep}, so there exists $C>0$ such that
$$
\big\langle \eta_x f(\eta_x ) \big\rangle_{L,N} \geq C \min_{n>K} f(n) \quad\mbox{for all $K$ large enough}\ .
$$
Then $f(n)\to\infty$ implies $\min_{n>K} f(n)\to\infty$ as $K\to\infty$, which proves the first statement.\\
Essentially the same argument works for the second statement, we have for all $K>0$ fixed
$$
\Big\langle \frac{\eta_x}{f(\eta_x )} \Big\rangle_{L,N} \leq\Big\langle \1_{\eta_x \leq K}\frac{\eta_x}{f(\eta_x )} \Big\rangle_{L,N} + \frac{1}{\min_{n>K} f(n)} \big\langle \1_{\eta_x >K}\eta_x \big\rangle_{L,N} \to \Big\langle \1_{\eta_x \leq K}\frac{\eta_x}{f(\eta_x )} \Big\rangle_{\rho}
$$
as $L,N\to\infty$, $N/L\to\rho$, because $\min_{n>K} f(n)$ diverges and $\langle \1_{\eta_x >K}\eta_x \big\rangle_{L,N}$ is uniformly bounded since it converges to $\rho -\rho_b$ as $K\to\infty$ \eqref{phasesep}. In that limit, the right-hand side converges to $\big\langle \eta_x /f(\eta_x )\rangle_\rho$ which implies
$$
\limsup_{L\to\infty ,N/L\to\rho}\Big\langle \frac{\eta_x}{f(\eta_x )} \Big\rangle_{L,N} \leq \Big\langle \frac{\eta_x}{f(\eta_x )} \Big\rangle_{\rho}\ .
$$
This implies in particular that $\liminf\limits_{L\to\infty ,N/L\to\rho}\Big\langle \frac{\eta_x}{f(\eta_x )}\1_{\eta_x >K} \Big\rangle_{L,N} \to 0$ as $K\to\infty$. Therefore we get the lower bound
\begin{align*}
\Big\langle \frac{\eta_x}{f(\eta_x )} \Big\rangle_{L,N} \geq \Big\langle \1_{\eta_x \leq K}\frac{\eta_x}{f(\eta_x )} \Big\rangle_{\rho}+\liminf_{L\to\infty ,N/L\to\rho}\Big\langle \frac{\eta_x}{f(\eta_x )}\1_{\eta_x >K} \Big\rangle_{L,N}\ ,
\end{align*}
which converges to $\big\langle \eta_x /f(\eta_x )\rangle_\rho$ as $K\to\infty$.
\end{proof}

This result implies in particular, that for condensing systems all higher moments $\langle \eta_x^a \rangle_{L,N}$ with $a>1$ diverge in the thermodynamic limit due to contributions from the condensed phase. Lower moments with $a<1$ converge to $\langle \eta_x^a \rangle_\rho$, and the first moment with $a=1$ is the boundary case, converging to a strictly larger value $\rho >\rho_b =\langle \eta_x \rangle_\rho$ than the bulk density. We stress again that we have only used Definition \ref{cdef} and weak convergence of single-site marginals of the canonical measures to derive these results. So they hold very generally, and do not depend on the existence of stationary product measures or any other particular structure.

\section{Some details on dynamics and Monte Carlo simulations\label{appb}}

Heuristic results for TA dynamics of the inclusion process \cite{cao2014dynamics} show that the equilibration time scales like $L/d$, and is dominated by a coarsening process with a transport limited mass exchange dynamics between isolated clusters: On a time scale of order $1$ the mass in the system concentrates on isolated cluster sites which are separated by at least one empty site. Each cluster of size $m$ then performs an effective totally asymmetric random walk with rate $dm$. So larger clusters move faster and overtake smaller ones, and during the overtake both clusters exchange mass. This leads to fluctuations in cluster sizes and drives the coarsening process, where smaller clusters disappear and the average cluster size grows as a power law in time. From the point of view of an individual cluster, coarsening determines the time scale $\tau_a$ on which it aggregates a macroscopic amount of mass, and on the fragmentation time scale $\tau_f$ it loses a non-zero mass fraction which forms a new cluster on a previously empty site. For TA dynamics, the latter only happens if during a step when a cluster extends over two sites (which takes only a time fraction of order $d$), a further particle breaks away, which happens again at rate proportional to $d$ (see discussion in \cite{cao2014dynamics} for more details). In summary both time scales are
\[
\tau_a =L/d\quad\mbox{and}\quad \tau_f = d^{-2}\ ,
\]
and we see that they agree exactly in the case $dL\to\alpha\in (0,\infty)$, leading to a balance of aggregation and fragmentation for macroscopic clusters at stationarity, and the interesting hierarchical structures of Theorem \ref{thm1}. If $dL\to 0$ then $\tau_a \ll\tau_f$ and the balance cannot be reached, rather the system saturates in a single remaining cluster consistent with complete condensation results Proposition \ref{propc}. On the other hand if $dL\to\infty$, fragmentation dominates with $\tau_f \ll \tau_a$ for macroscopic clusters, and a balance is reached at sizes of scale $1/d$ instead (consistent with Theorem \ref{interthm}), which includes the case of no condensation with $d=O(1)$. This heuristic provides useful insight on the level of the dynamics into our rigorous results which only depend on the form of the stationary distributions \eqref{canme}, and also implies that TA dynamics have to be simulated on times of order $\tau_a = L/d$ to reach stationarity.

\RestyleAlgo{boxruled}
\begin{algorithm}
\caption{Inclusion process \eqref{ipgen} on a complete graph (CG dynamics)\label{alg1}}
\textbf{Parameters} $L$ size of lattice $\Lambda$; $N$ \# of particles; $d>0$; $t$ simulation time\;
\textbf{Initialize} particle locations $\sigma_i \sim U(\Lambda )$, $i=1,\ldots ,N$ i.i.d. uniform\;
\smallskip
 \While{$s<t$}{
  pick particle $i\sim U\big( [1..N]\big)$ uniformly at random\;
  \eIf{$R\sim U\big( [0,1)\big)<dL/(dL+N)$}{
   $\sigma_i \leftarrow U(\Lambda )$;
   }{
  pick particle $j\sim U\big( [1..N]\big)$ uniformly at random\;
  $\sigma_i \leftrightarrow \sigma_j$ exchange positions\;
  }
 $s\leftarrow s+\frac{1}{N(dL+N)}$;
  }
 \textbf{Output} $\eta_x =\sum_i \delta_{\sigma_i ,x}$ for $x=1,\ldots ,L$\;
 \bigskip
\end{algorithm}

A similar argument can be made for the complete graph geometry, where the dynamics is entirely different. Cluster sites are in direct contact, and exchange single particles with a rate of order $m^2 /L$, where we understand $m\gg 1$ to be a 'typical' cluster size. Since the exchange is symmetric, it takes of order $m^2$ exchange events to change cluster sizes by a finite fraction, leading to
\[
\tau_a =\frac{L}{m^2} m^2 =L\quad\mbox{and}\quad\tau_f =\frac{1}{dm} m=\frac{1}{d}\ .
\]
The fragmentation time scale $\tau_f$ follows since particles jump onto empty sites with rate $dm$ and of order $m$ jumps are needed to fragment a finite fraction of a cluster's mass. Here we used that due to $m\gg 1$ cluster sites only cover a vanishing volume fraction. Even though both time scales are different from TA dynamics, an aggregation fragmentation balance is again reached for $dL\to\alpha$. Since we only care about the mass distribution and not the spatial location of clusters, equilibration time is now faster of order $\tau_a =L$. This is a crucial difference to TA dynamics, where the coarsening process is transport limited and clusters have to move in order to exchange particles. Due to the particular form of the jump rates for the inclusion process \eqref{ipgen}, CG dynamics can be implemented in a rejection-based algorithm summarized in Alg. \ref{alg1}, and this provides a very simple and efficient way to produce Monte Carlo samples from the distribution $\pi_{L,N}$ \eqref{canme}.


\end{document}